\documentclass{amsart}
\usepackage{amscd,amssymb,amsmath,amsfonts,amsthm}
\usepackage{tasks}
\usepackage{longtable}
\usepackage[shortlabels]{enumitem}
\usepackage[a4paper,top=3cm, bottom=3cm, left=3cm, right=3cm]{geometry}
\begin{document}

\vfuzz2pt 
\newcommand\ddfrac[2]{\frac{\displaystyle #1}{\displaystyle #2}}

 \newtheorem{thm}{Theorem}[section]
 \newtheorem{cor}{Corollary}[section]
 \newtheorem{lem}{Lemma}[section]
 \newtheorem{prop}{Proposition}[section]
 \theoremstyle{definition}
 \newtheorem{defn}{Definition}[section]
 \theoremstyle{remark}
 \newtheorem{rem}{Remark}[section]
 \numberwithin{equation}{section}
\newcommand{\CC}{\mathbb{C}}
\newcommand{\KK}{\mathbb{K}}
\newcommand{\ZZ}{\mathbb{Z}}
\newcommand{\RR}{\mathbb{R}}
\def\a{{\alpha}}

\def\b{{\beta}}

\def\d{{\delta}}

\def\g{{\gamma}}

\def\l{{\lambda}}

\def\gg{{\mathfrak g}}
\def\cal{\mathcal }

\title{Abelian Subalgebras and Ideals of Maximal Dimension in Poisson algebras}

\author{A.  Fern\'andez Ouaridi}
\address{Amir Fern\'andez Ouaridi
\newline \indent Dpto. de Matem{\'a}ticas, Universidad de C\'adiz, Puerto Real (Spain)}
 \email{amir.fernandez.ouaridi@gmail.com}

\author{R.M. Navarro}
\address{Rosa Mar{\'\i}a Navarro \newline \indent
Dpto. de Matem{\'a}ticas, Universidad de Extremadura, C{\'a}ceres
 (Spain) }
\email{rnavarro@unex.es}

\author{D. A. Towers}
\address{David A. Towers
\newline \indent Dept. of Mathematics and Statistics,
Lancaster University, Lancaster LA1 4YF (England)}
 \email{d.towers@lancaster.ac.uk}

\thanks{This work was  supported  by Ministerio de Econom{\'\i}a y Competitividad (Spain), grant PID2020-115155GB-I00 (European FEDER support included,  EU), by {Junta de Andalucía, Consejería de Universidad, Investigación e Innovación: ProyExcel\_00780 ``Operator theory: An interdisciplinary approach"} and by the Spanish Government through the Ministry of Universities grant `Margarita Salas', funded by the European Union - NextGenerationEU}

\begin{abstract} 

This paper studies the abelian subalgebras and ideals of maximal dimension of Poisson algebras $\mathcal{P}$ of dimension $n$. We introduce the invariants $\alpha$ and $\beta$ for Poisson algebras, which correspond to the dimension of an abelian subalgebra and ideal of maximal dimension, respectively. We prove that these invariants coincide if $\alpha(\mathcal{P}) = n-1$. We characterize the Poisson algebras with $\alpha(\mathcal{P}) = n-2$ over arbitrary fields. In particular, we characterize Lie algebras $L$ with $\alpha(L) = n-2$. We also show that $\alpha(\mathcal{P}) = n-2$ for nilpotent Poisson algebras implies $\beta(\mathcal{P})=n-2$. Finally, we study these invariants for various distinguished Poisson algebras, providing us with several examples and counterexamples.

\bigskip 

{\it 2020MSC}: { 17A30,
17B30,
17B63.}


\end{abstract}

\maketitle

\section{Introduction}

{ Poisson algebras naturally arise in different areas of mathematics and physics, including algebraic geometry, quantization theory, quantum groups, classical mechanics, quantum mechanics, general relativity, geometrical optics or quantization theory. A Poisson algebra is an associative commutative algebra together with a Lie bracket that satisfies a Leibniz compatibility rule, making this class an algebraic variety. 
In a broad setting, the study of a subvariety of Poisson algebras relies on the identification of the rigid algebras which determine its irreducible components. 
The rigid algebras of the variety of Poisson algebras and nilpotent Poisson algebras were obtained in small dimension, see \cite{pan, pa3}. However, as higher dimensions are considered the complexity of the problem increases rapidly. Therefore, developing new tools is required. 
In this sense, considering the invariants $\alpha$ and $\beta$, corresponding respectively to the dimensions of abelian subalgebras and ideals of maximal dimension, becomes very useful.

The study of abelian subalgebras and ideals of maximal dimension has been pursued across various classes of non-associative algebras, including Lie algebras \cite{BC12, Towers13, Towers27}, Leibniz algebras \cite{Towers}, and Zinbiel algebras \cite{TowersZ}. This investigation holds significance for multiple reasons, such as determining the structural properties within this variety of algebras, addressing specific classification problems, or identifying rigid algebras within a given variety. Principal findings concerning Lie algebras (and Leibniz algebras) are summarized as follows: Let $L$ denote a Lie algebra (or a Leibniz algebra) of dimension $n$. Maximal subalgebras that are abelian of Leibniz algebras have codimension one, see \cite{BC12, Towers}. For solvable Lie algebras over an algebraically closed field of characteristic zero, it is known that $\alpha(L) = \beta(L)$, see \cite{BC12}. Explicit computation of  $\alpha$ and $\beta$ for small-dimensional complex Lie and Leibniz algebras has been provided \cite{Ceballos, Towers}. Over an arbitrary field of characteristic $p\neq2$, if $L$ is a Leibniz algebra and $\alpha(L)$ equals $n-1$, then $\beta(L)$ equals $n-1$. Additionally,  if $L$ is a supersolvable Lie algebra or a nilpotent Leibniz algebra such that $\alpha(L)=n-2$, then $\beta(L)=n-2$, see  \cite{Towers}. Furthermore, if $L$ is a nilpotent or a supersolvable Lie algebra such that $\alpha(L)=n-3$, then $\alpha(L)=n-3$ (for $p\neq 2$), see \cite{Towers13, Towers27}. If $L$ is a nilpotent Lie algebra  and $\alpha(L)=n-4$, then $\beta(L)=n-4$ (for $p\neq 2, 3, 5$), see \cite{Towers27}.  
 The existence of abelian subalgebras of finite codimension in a Lie algebra determine the existence of certain customary identities satisfied by the symmetric Poisson algebra $s(L)$, as was noted by Farkas \cite{Farkas1, Farkas2}. 
Following the spirit of these results, we approach Poisson algebras, which can be interpreted as a simultaneous generalisation of Lie algebras and associative commutative algebras.

This paper is divided into five sections.  Along the second section, we review the definitions and elementary results that will be useful in the subsequent sections. In the third section,  we study the maximal subalgebras which are abelian in a Poisson algebra $\mathcal{P}$.  In particular,  if
the field is algebraically closed, we show that they have codimension one and that $\mathcal{P}$ is solvable.   Throughout the forth section,  we introduce the dialgebra version of the invariants $\alpha$ and $\beta$, for Poisson algebras. We study the cases $\alpha(\mathcal{P})=n-1$ and $\alpha(\mathcal{P})=n-2$, obtaining a characterization of the Poisson algebras with $\alpha(\mathcal{P})=n-2$.   Finally,  the fifth section is devoted to the study of the invariants $\alpha$ and $\beta$ for some important families of
finite-dimensional Poisson algebras.

}

\section{Basic concepts and preliminaries}

A dialgebra $(\mathcal{A}, \cdot, [\cdot, \cdot])$ is a vector space $\mathcal{A}$ endowed with two multiplications $\cdot$ and $[\cdot, \cdot]$ that are not necessarily associative. Some popular classes of dialgebras are Lie-Yamaguti algebras, Gerstenhaber algebras, Nambu-Poisson algebras, Novikov-Poisson algebras, Gelfand-Dorfman algebras, and many others. Also, Poisson algebras, the object of study of this paper.

\begin{defn}\rm
A {Poisson algebra} $\mathcal{P}$ is a dialgebra $(\mathcal{P}, \cdot, [\cdot, \cdot])$ such that $\mathcal{P}_A:=(\mathcal{P}, \cdot)$ is an associative commutative algebra, $\mathcal{P}_L:=(\mathcal{P}, [\cdot, \cdot])$ is a Lie algebra and they satisfy the compatibility identity for $x,y,z\in \mathcal{P}$ given by
\begin{equation*}
[x\cdot y, z]=[x,z]\cdot y + x\cdot[y,z] \quad \quad \textrm{(Leibniz rule).}
\end{equation*}
\end{defn}

We say a dialgebra is trivial if one of the multiplications is zero. A trivial Poisson algebra is just an associative commutative algebra or a Lie algebra, so we may think of Poisson algebras as a simultaneous generalization of Lie algebras and associative commutative algebras. Some of our results only hold when both multiplications are non-zero and others also hold in particular for Lie algebras and associative commutative algebras. 
Every algebra in this paper is finite-dimensional, unless we say otherwise. All the vector spaces, algebras and linear maps are considered over an arbitrary field $\mathbb{F}$ of characteristic $p$. The multiplication $\cdot$ will be denoted by concatenation.
For a Poisson algebra $\mathcal{P}$, we denote by $P_x$ and $Q_x$ the maps in $\textrm{End}(\mathcal{P})$ given by $P_x(y) = xy$ and $Q_x(y) = [x, y]$, for $x, y \in \mathcal{A}$. Let us recall some definitions.

\begin{defn}
A {subalgebra}  of a dialgebra $\mathcal{A}$ is a linear subspace $A$ closed by both multiplications, that is $A \cdot A + [A,A]    \subset A$.
A subalgebra $I$ of $%
\mathcal{A}$ is  an {ideal } if $I \cdot \mathcal{A} + \mathcal{A}\cdot I +[I,\mathcal{A}] + [\mathcal{A}, I]    \subset I$. An abelian subalgebra (or ideal) is a subalgebra (or ideal)  $A$ such that $A \cdot A + [A,A]    = 0$. In this paper, abelian stands for trivial.
\end{defn}

{Given a dialgebra $(\mathcal{A}, \cdot, [\cdot, \cdot])$, recall the derived
sequence of subspaces. For $n \geq 0$,  define
$$\mathcal{A}^{\left. 0\right) }:=\mathcal{A} \quad \quad \quad \mathcal{A}^{n+1)} =  \mathcal{A}^{n)}\cdot\mathcal{A}^{n)} + [\mathcal{A}^{n)}, \mathcal{A}^{n)}].$$

\begin{defn} 
\rm
A dialgebra $(\mathcal{A}, \cdot, [\cdot, \cdot])$ is {solvable} if there exist $m\geq 0$ such that $\mathcal{A}^{\left. m\right) }=0$. 
\end{defn}

Moreover, the lower central series is the sequence
$$\mathcal{A}^{\left( 0\right) }:=\mathcal{A} \quad \quad \quad \mathcal{A}^{(n+1)} = \sum_{i=1}^{n}( \mathcal{A}^{(i)}\cdot\mathcal{A}^{(n+1-i)} + [\mathcal{A}^{(i)}, \mathcal{A}^{(n+1-i)}]).$$

\begin{defn} 
\rm
A dialgebra $(\mathcal{A}, \cdot, [\cdot, \cdot])$ is {nilpotent} if there exists $m\geq 0$ such that $\mathcal{A}^{\left( m\right) }=0$. 
\end{defn}

Note that if a dialgebra is solvable (resp. nilpotent), then each of the multiplications is solvable (resp. nilpotent). Also, if $\mathcal{A}$ is a Poisson algebra, then we have 
$
\mathcal{A}^{\left( n+1\right) }=  \mathcal{A}^{\left( n\right)}  \cdot  \mathcal{A}   +
[ \mathcal{A}^{\left( n\right)} ,  \mathcal{A}  ]$.

Next, we introduce the definition of the normalizer of a subalgebra of a dialgebra. 

\begin{defn}
Given a dialgebra $\mathcal{A}$ and a subalgebra $A$, the normalizer of $A$ is the set
$$N(A) = \left\{x\in \mathcal{P}: A x + x A\subseteq A \textrm{ and } [A, x] + [x, A]\subseteq A\right\}.$$
If $N(A)=A$, we say that $A$ is self-normalizing.
\end{defn}

The well-known result for non-associative algebras that asserts that the normalizer of any proper subalgebra $A$ of a nilpotent algebra satisfies $N(A)\neq A$, also holds for dialgebras.

\begin{prop}\label{normnil}
Let $\mathcal{A}$ be a nilpotent dialgebra. Given any proper subalgebra $A$, then $N(A)\neq A$.
\end{prop}
\begin{proof}
    
    
    {
    
    Since $A$ is a proper subalgebra of  $\mathcal{A}$  there is an index $k$ such that  $\mathcal{A}^{(k+1)}\subset A$ but $\mathcal{A}^{(k)}\nsubseteq A$. Take any $x \in \mathcal{A}^{(k)}\setminus A$. Then, we have $xA + Ax\subseteq \mathcal{A}^{(k+1)}\subset  A$ and $[x, A]+[A, x]\subseteq  \mathcal{A}^{(k+1)}\subset A$, implying that $x\in N(A)$. As $x \notin A$ we conclude $A \subsetneq N(A)$, proving the result.} 
\end{proof}

The normalizer of a subalgebra $A$ of a Poisson algebra is a subalgebra containing $A$.

\begin{prop}
Let $\mathcal{P}$ be a Poisson algebra and let $A$ be a subalgebra of $\mathcal{P}$. Then its normalizer $N(A)$ is a subalgebra of $\mathcal{P}$ and $A\subseteq N(A)$. 
\end{prop}
\begin{proof}
    Given $x, y \in N(A)$, we have to prove that $(xy)A\subseteq A$, $[xy, A]\subseteq A$,  $[x, y] A\subseteq A$ and $[[x, y], A]\subseteq A$. The first inclusion follows by the associativity. The second and the third inclusions require the Leibniz rule.
    $$[xy, A]\subseteq [x, A]y + [y, A]x \subseteq A, \quad \quad 
    [x, y] A\subseteq  [xA, y] + [A, y] x\subseteq A.$$
    The last inclusion follows by the Jacobi identity $[[x, y], A]\subseteq [[A, x], y] + [[y, A], x]\subseteq A$. Therefore, $N(A)$ is a subalgebra.
    Finally, clearly $N(A)$ contains $A$.
\end{proof}

\begin{rem}\label{remark1}
Let $A$ be a subalgebra of a Poisson algebra $\mathcal{P}$, then $N(A)=\mathcal{P}$ if and only if $A$ is an ideal. Moreover, if $A$ is a maximal subalgebra, then either $A$ is an ideal (i.e. $N(A)=\mathcal{P}$) or $A$ is self-normalizing (i.e. $N(A)=A$). Furthermore, if $\mathcal{P}$ is nilpotent then $N(A)=\mathcal{P}$.
\end{rem}

\section{On the maximal subalgebras of Poisson algebras}

In this section, we study the maximal subalgebras which are abelian in a Poisson algebra $\mathcal{P}$. If the field is algebraically closed, we show that they have codimension one and that $\mathcal{P}$ is solvable.

\begin{prop}\label{prop22}
     Let $\mathcal{P}$ be a Poisson algebra with an abelian subalgebra $A$ of maximal dimension. Then, if it is an ideal of $\mathcal{P}_L$, it is also an ideal of $\mathcal{P}$.
 \end{prop}
\begin{proof}
    If $A$ is not an ideal of $\mathcal{P}_A$, then there is some $a\in A$ and $x\in \mathcal{P}$ such that $ax\not \in A$. Consider $A'=A+ \mathbb{F}(ax)$. It is clear that $A'$ is an abelian subalgebra of $\mathcal{P}_A$. Also, we have $[A, ax]\subseteq a[A, x]\subseteq\left\{0\right\}$, using the Leibniz rule. Hence, $A'$ is an abelian subalgebra of $\mathcal{P}$, which is in contradiction with the maximality of $A$. 
\end{proof}

\begin{rem}
An abelian ideal of $\mathcal{P}_L$ is not always an abelian subalgebra of $\mathcal{P}_A$. Moreover, the existence of an abelian ideal of $\mathcal{P}_L$ (or $\mathcal{P}_A$) of dimension $s$, does not guarantee the existence of an abelian subalgebra of $\mathcal{P}$ of dimension $s$. The algebra $\mathcal{P}_{3,20}$ in Table \ref{tab1} illustrates these facts. Furthermore, the nilradical (and radical) of $\mathcal{P}_L$ is not necessarily an ideal of $\mathcal{P}$. For example, consider the algebra $\mathcal{P}_{3,18}$.
\end{rem}

Note that the condition of $\mathcal{P}_A$ being nilpotent guarantees the existence of an abelian subalgebra. If $z$ is in the annihilator of $\mathcal{P}_A$, then $z$ generates a one-dimensional subalgebra of $\mathcal{P}$. On the other hand, if $\mathcal{P}_A$ has an idempotent, then 
$\mathcal{P}$ has a subalgebra of dimension one generated by it. {Since any associative algebra is
either nilpotent or contains an idempotent, we have the following.}

\begin{lem}
    If  $\mathcal{P}$ is a Poisson algebra of dimension $n$, then it has a subalgebra of dimension one.
\end{lem}

\begin{prop}\label{propac}
Let  $\mathcal{P}$ be a Poisson algebra of dimension $n$. If $\mathcal{P}$ has a maximal subalgebra $A$ which is an ideal, then $\textrm{dim}(A) = n-1$. 
\end{prop}
\begin{proof}
    By the previous lemma, we have that $\mathcal{P}/A$ has a subalgebra $B$ of dimension one. If $B'$ denotes its preimage, then $B'+A = \mathcal{P}$ by the maximality of $A$.  Consequently, $\textrm{dim}(A) = n-1$.
\end{proof}

\begin{thm}\label{thmpa1}
    Let $\mathcal{P}$ be a Poisson algebra of dimension $n$ over an algebraically closed field. If $\mathcal{P}$ has a maximal subalgebra $A$ which is abelian, then $\textrm{dim}(A) = n-1$. Moreover, if $N(A)=A$, then $\mathcal{P}_A$ is trivial.
\end{thm}

\begin{proof}

Let $A$ be a maximal subalgebra which is abelian. If $A$ is an ideal, then the statement follows by Proposition \ref{propac}. Otherwise, we have that $N(A) = A$, by virtue of Remark \ref{remark1}. Observe that 
$P_{a} P_{b} = P_{b} P_{a}$ by the associativity, $Q_{a}Q_{b} = Q_{b}Q_{a}$ by the Jacobi identity  and $P_{a}Q_{b} = Q_{b}P_{a}$ by the Leibniz rule for any $a,b \in A$. Hence, the maps $P_{a}$ and $Q_{b}$ for any $a,b \in A$ conmute.  In fact, we have $P_{a} P_{b}=0$, so the maps $P_a$ are nilpotent. Denote $M(A):=\left\{P_{a}, Q_{a} : a\in A\right\}$. 
Let $\mathcal{P}= V_0 \oplus V_1$ be the Fitting decomposition of $\mathcal{P}$ with respect to the maps in $M(A)$, see \cite[Lemma 1, p. 38]{jacobson}. Recall that
$$V_0:= \left\{x\in \mathcal{P}:\varphi^r(x)=0 \textrm{ for any }\varphi \in M(A) \right\} \textrm{ and } V_1:=\sum_{\varphi\in M(A)} V_{1,\varphi}, \textrm{ where } V_{1,\varphi}= \cap_{k=1}^{\infty} \varphi^k(\mathcal{P}).$$

Since $N(A)=A$, we have that $V_0=A$. 
Indeed, on the one hand, it is clear that $A\subseteq V_0$, by the abelianity of $A$. On the other hand, suppose $x\in V_0$ and $x\not \in A$. Then, there is some $a_1\in A$ such that $\varphi_{a_1}(x)\not\in A$, where $\varphi_{a_1}$ is either $P_{a_1}$ or $Q_{a_1}$. We can repeat this indefinitively, so $\varphi_{a_k}\ldots\varphi_{a_1}(x) \not\in A$ for suitables $a_1, \ldots, a_k\in A$. Since $A$ is finite-dimensional and $x\in V_0$, there is some $N$ such that $\varphi_{a_N}\ldots\varphi_{a_1}(x)\not\in A$ with $\varphi \varphi_{a_N}\ldots\varphi_{a_1}(x)=0$ for any $\varphi\in M(A)$. But then, since $N(A)=A$, we have $\varphi_{a_N}\ldots\varphi_{a_1}(x)\in A$, which is a contradiction. Hence, $A= V_0$ and  $\mathcal{P}= A \oplus V_1$, so we have $V_1\neq \left\{0\right\}$.

The maps in $M(A)$ are simultaneously triangulable in $V_1$; it follows that there exists some element $x\not\in A$ such that $P_{a}(x)=\lambda_{a}x$ and $Q_{a}(x)=\mu_{a}x$ for any $a\in A$, where $\lambda_a, \mu_a \in \mathbb{F}$. Moreover, $\lambda_a=0$ for any $a\in A$, because $P_{a}$ is nilpotent. If there is some $a\in A$ with $\mu_a\neq 0$, then $\mu_a^{-1}Q_a(x)=x$ and we can write $xx= \mu_a^{-1}Q_a(x)x = \mu_a^{-1} [a,x]x = \mu_a^{-1} ([xa, x] - a[x,x]) = 0$. We conclude that $\mathcal{P}_A$ is trivial and $A+\mathbb{F} x$ is a subalgebra of $\mathcal{P}$ which is strictly containing $A$, and therefore that it must be $\mathcal{P}$. On the contrary, if we assume $\mu_a= 0$ for any $a\in A$, then $\varphi(x) = 0$ for any $\varphi\in M(A)$. This implies that $x\in N(A)=A$, which is a contradiction.
\end{proof}

\begin{cor}\label{nilpcodim1cor}
    Let $\mathcal{P}$ be a non-trivial Poisson algebra of dimension $n$ over an algebraically closed field. If $\mathcal{P}$ has a maximal subalgebra $B$  which is abelian, then $\mathcal{P}^{1)}\subseteq B$, $\mathcal{P}$ is solvable and $\mathcal{P}^{2)}=0$. 
\end{cor}
\begin{proof}
    By Theorem \ref{thmpa1}, $B$ is an abelian ideal of codimension one.
    Hence, $\mathcal{P} = B \oplus \mathbb{F} x$.  Then $\mathcal{P}^{1)} = [B \oplus \mathbb{F} x, B \oplus \mathbb{F} x] + (B \oplus \mathbb{F} x) (B \oplus \mathbb{F} x) \subseteq B + \mathbb{F}x^2$. Suppose $x^2\not \in B$. Then $x^2=\lambda x + B$ with $\lambda\neq0$, so we have $x= \lambda^{-1}x^2+B$. Since $\mathcal{P}$ is not trivial, there is $b\in B$ such that $[x,b]\neq0$. Now, observe that $[x, b] = [\lambda^{-1}x^2+B, b] = 2 \lambda^{-1} x[x, b]$ and $x[x, b]= \frac{\lambda}{2} [x, b]$. Moreover,  we have   
    $2 x[x, b] = 2(\lambda^{-1}x^2+B)[x, b] = 2\lambda^{-1} x^2[x, b] = x [x, b]$, so $x[x, b]=0$ and $[x, b]=0$, which is a contradiction. Hence, $x^2\in B$. It follows that     
    $\mathcal{P}^{2)} = [\mathcal{P}^{1)}, \mathcal{P}^{1)}] + \mathcal{P}^{1)}\mathcal{P}^{1)}\subseteq [B, B]+ BB = 0$. 
\end{proof}

\begin{rem}\label{nilpcodim1}
    The nilpotency is not guaranteed: see, for example, the algebra $\mathcal{P}_{3,14}$ in Table \ref{tab1}. The condition on the field can be dropped easily if we assume we have an abelian subalgebra of codimension one.    
    If we remove the condition of not being trivial, the result does not hold as we can consider the two dimensional algebra spanned by $e_1, e_2$ given by $e_1^2 =e_1$, which is a Poisson algebra. It has a maximal subalgebra $A=\textrm{span}(e_2)$ which is abelian, but it is not solvable. 
    Moreover, if $\mathbb{F}$ is the complex field we can prove the following result. 
\end{rem}

{Let $L_1(\Gamma)$ be the three-dimensional simple Lie algebra with basis $e_1, e_2, e_3$ and products $[e_1, e_2] = e_2, [e_1, e_3]=\gamma e_2-e_3, [e_2, e_3]=e_1$ over a field of arbitrary characteristic $p$ studied in \cite{amayo1, amayo2}. If $p\neq 2$, then $L_1(\Gamma)\cong \mathfrak{sl}_2$ as we can choose the basis $E_1 = e_1$, $E_2=e_2$ and $E_3=-\frac{\gamma}{2}e_2 + e_3$ to obtain $[E_1, E_2] = E_2, [E_1, E_3]=-E_3, [E_2, E_3]=E_1$. Otherwise, $L_1(\Gamma)$ is a parametric family of non-isomorphic simple Lie algebras. }

{
\begin{lem}\label{sl2+}
    Let $\mathcal{P}$ be a non-trivial Poisson algebra of dimension $n$ with an abelian subalgebra of codimension two over a field $\mathbb{F}$ of characteristic $p$. If $\mathcal{P}_L\cong L_1(\Gamma)\oplus \mathbb{F}^{n-3}$, then $\mathcal{P}$ is isomorphic to the algebra $\mathfrak{p}_n(\gamma) = \mathfrak{p}_4(\gamma) \oplus \mathbb{F}^{n-4}$, where $\mathfrak{p}_4(\gamma)$ is the Poisson algebra with basis $e_1, \ldots, e_4$ and multiplication given by 
    $$\mathfrak{p}_4(\gamma):\left\{ \begin{array}{cll}
 [e_1, e_2] = e_2, [e_1, e_3]=\gamma e_2-e_3, [e_2, e_3]=e_1,\\ 
e_1e_1 = e_2 e_3 = e_4, \, e_3e_3 = \gamma e_4.
\end{array}\right.$$
Moreover, if $p\neq 2$, then $\mathfrak{p}_4(\gamma)\cong \mathfrak{p}_4(0)$ for any $\gamma\in \mathbb{F}$.
\end{lem}
\begin{proof}
    Let $A$ be an abelian subalgebra of codimension two of $\mathcal{P}$ and suppose $\mathcal{P}_L=L_1(\Gamma)\oplus \mathbb{F}^{n-3}$. Fix a basis $e_1, e_2, e_3$ for $L_1(\Gamma)$ and $a_1, \ldots, a_{n-3}$ for $\mathbb{F}^{n-3}$ such that the multiplication is given by $[e_1, e_2] = e_2, [e_1, e_3]=\gamma e_2-e_3, [e_2, e_3]=e_1$. It is easy to see that any abelian subalgebra of codimension two of $L_1(\Gamma)\oplus \mathbb{F}^{n-3}$ is of the form $A=\textrm{span}(v:= \lambda e_1 + \mu e_2 + \eta e_3, a_1, \ldots, a_n)$. Assume $A$ is an abelian subalgebra of $\mathcal{P}$. 
    First, we show that $e_ia_j=0$. Using $va_j =\lambda e_1a_j + \mu e_2a_j + \eta e_3a_j=0$ we obtain the equations
    $$\lambda e_2a_j= \lambda [e_1, e_2] a_j = \lambda [e_1a_j, e_2] = - \mu [e_2a_j, e_2] - \eta [e_3a_j, e_2] =  - \eta a_j[e_3, e_2]= \eta  e_1a_j,$$
    $$-\eta \gamma e_2 a_j + \eta e_3a_j=-\eta [e_1, e_3] a_j = -\eta [e_1, e_3a_j] = \lambda [e_1, e_1a_j] +\mu [e_1, e_2a_j]=\mu a_j[e_1, e_2]=\mu e_2 a_j,$$  
    $$\mu e_1a_j= \mu[e_2, e_3] a_j = \mu[e_2a_j, e_3] = -\lambda[e_1a_j, e_3] - \eta [e_3 a_j, e_3] = -\lambda a_j [e_1, e_3]= -\lambda \gamma e_2a_j + \lambda e_3a_j.$$
    Now, we discuss, depending on the parameters of $v$, the equations above. 
    \begin{itemize}
        \item[--] If $\lambda\neq0$, then we have $e_2a_j = \lambda^{-1}\eta e_1a_j$ and $e_3 a_j = \lambda^{-1}\mu e_1 a_j + \gamma e_2 a_j$. Therefore, $e_3a_j = -[e_1, e_3a_j] + \gamma e_2 a_j = -\lambda^{-1}\mu [e_1, e_1 a_j] - \gamma [e_1, e_2 a_j] + \gamma e_2 a_j=0$. Also, we have 
        $e_1a_j= [e_2, e_3]a_j = [e_2, e_3a_j] = 0$ and $e_2a_j=0$.

        \item[--] If $\lambda=0$ and $\mu\neq0$, then $e_1a_j=0$, $e_2a_j = [e_1a_j, e_2]=0$ and $e_3a_j = -[e_1a_j, e_3] + \gamma e_2 a_j=0$. 
        
        \item[--] If $\lambda=\mu=0$ and $\eta\neq0$, then $e_1a_j=0$, $e_2a_j = [e_1a_j, e_2] = 0$ and $e_3a_j=0$. 
    \end{itemize}

    Now, by the Leibniz rule, the associativity and $e_ia_j=0$, it can be proved that there is some $a\in \textrm{span}(a_1, \ldots, a_{n-3})$ such that $e_1e_1 = e_2 e_3 = a$, $e_3e_3 = \gamma a$ and $e_i e_j = 0$ otherwise.
    
    \medskip

\noindent\underline{Claim}: $e_2e_2=e_1e_2=e_1e_3=0$, $e_2e_3 = e_1e_1 = a$ and $e_3e_3=\gamma a$, where $a\in \textrm{span}(a_1, \ldots, a_{n-3})$. 
Indeed, since $[e_1, e_1 e_1] = 0$, $[e_1, e_1e_2] = e_1[e_1, e_2] = e_1e_2$, $[e_1, e_1e_3] = e_1[e_1, e_3] = -e_1e_3+ \gamma e_1 e_2$, $[e_1, e_2e_3] = e_2[e_1, e_3] + e_3[e_1, e_2] = \gamma e_2 e_2$ and $e_2e_2=[e_1, e_2]e_2 = [e_1e_2, e_2]$, we can write
$$e_1e_1=\lambda_1 e_1 + a, \quad e_1e_2 = \lambda_2 e_2 + \lambda_3 e_3,  \quad e_1e_3 = \lambda_4 e_2 + \lambda_5 e_3,  \quad e_2e_2 = -\lambda_3 e_1, \quad e_2e_3 = \lambda_6e_1 + a'.$$
for $\lambda_i\in\mathbb{F}$ and $a, a' \in \textrm{span}(a_1, \ldots, a_{n-3})$. 
But now, we have $e_1e_2 = e_2[e_2, e_3]=[e_2, e_2e_3] = \lambda_6[e_2, e_1] = - \lambda_6 e_2$, so $\lambda_3=0$ and $\lambda_{6} = -\lambda_2$. Also, by $\lambda_2 \gamma e_2 - \lambda_2 e_3 = \lambda_2[e_1, e_3] = - [e_2e_3, e_3] = -e_3[e_2, e_3] = - e_1e_3 = -\lambda_4 e_2- \lambda_5 e_3$, we obtain $\lambda_4 = -\lambda_2 \gamma$ and $\lambda_5= \lambda_2$.
 
 Moreover, $\lambda_1 e_1 = e_1e_1 - a = [e_2e_1, e_3]- e_2[e_1, e_3] - a = \lambda_2 e_1 - a - \gamma e_2e_2 + e_2 e_3 =- a + a'$, implying $a=a'$ and $\lambda_1 = 0$.
Further, $\lambda_2 a = \lambda_2 e_1e_1 = - (e_2e_3)e_1 = - e_2 (e_3e_1) = \lambda_2 e_2e_3$. Hence, $\lambda_2=0$. Finally, we have $0 = [e_1e_3, e_3] = e_3[e_1, e_3] = \gamma e_3e_2 - e_3e_3$. The claim is proved.

\medskip

It follows that the algebra $\mathcal{P}_A$ is given by the products $e_1e_1 = e_2 e_3 = a$ and $e_3e_3 = \gamma a$. Since $\mathcal{P}$ is non-trivial, we have $e_4:=a\neq0$, obtaining the algebra $\mathfrak{p}_n(\gamma)$ in the statement. The verification of $\mathfrak{p}_n$ being a Poisson algebra is straightforward. Finally, if $\textrm{char}(\mathbb{F})\neq 2$, we can choose the basis with $E_1 = e_1$, $E_2=e_2$, $E_3=-\frac{\gamma}{2}e_2 + e_3$, showing that $\mathfrak{p}_4(\gamma)\cong \mathfrak{p}_4(0)$. 
\end{proof}

\begin{prop}\label{acod2}
     Let $\mathcal{P}$ be a non-trivial complex Poisson algebra of dimension $n$. If $\mathcal{P}$ has an abelian subalgebra of codimension two, then $\mathcal{P}_L$ is solvable or $\mathcal{P}_L\cong \mathfrak{sl}_2\oplus \mathbb{C}^{n-3}$ and $\mathcal{P}$ is isomorphic to the algebra $\mathfrak{p}_n = \mathfrak{p}_4 \oplus \mathbb{C}^{n-4}$.
\end{prop}
\begin{proof}
    Let $A$ be an abelian subalgebra of codimension two in $\mathcal{P}$. If $A$ is not an abelian subalgebra of maximal dimension, then there is an abelian subalgebra of codimension one and the result follows by Corollary \ref{nilpcodim1cor}. Assume $A$ is an abelian subalgebra of maximal dimension. Then, if $\mathcal{P}_L$ has an abelian subalgebra of codimension one it is clear that $\mathcal{P}_L$ is solvable, see \cite{BC12}. Now, suppose it does not have one. Then we have two possibilities according to \cite[Proposition 4.1]{BC12}. Either $\mathcal{P}_L$ is solvable or it is isomorphic to $\mathfrak{sl}_2\oplus \mathbb{C}^{n-3}$. By Lemma \ref{sl2+}, the second case together with  $A$ being an abelian subalgebra of codimension two, imply that the Poisson algebra $\mathcal{P}$ is isomorphic to $\mathfrak{p}_n$.   
\end{proof}
}

\begin{rem}
    Note that $\mathcal{P}$ in the previous result is not always solvable when $\mathcal{P}_L$ is solvable. In fact, $\mathcal{P}_A$ can be unital. See $\mathcal{P}_{3,18}$ in Table~\ref{tab1}. 
\end{rem}


\section{The invariants $\alpha$ and $\beta$ for Poisson algebras}

This section introduces the functions $\alpha$ and $\beta$ for Poisson algebras and considers the case in which the given dialgebra has an abelian subalgebra of codimension one or two. 

\begin{defn}
    Let $\mathcal{P}$ be a Poisson algebra. We define
$$\alpha(\mathcal{P})= max \{ dim(A) \mid A \mbox{ is an abelian subalgebra of } \mathcal{P}\},$$
$$\beta(\mathcal{P})= max \{ dim(I) \mid I \mbox{ is an abelian  ideal of } \mathcal{P}\},$$
$$\alpha_A(\mathcal{P})= max \{ dim(A) \mid A \mbox{ is an abelian subalgebra of } \mathcal{P}_A\},$$
$$\beta_A(\mathcal{P})= max \{ dim(I) \mid I \mbox{ is an abelian  ideal of } \mathcal{P}_A\},$$
$$\alpha_L(\mathcal{P})= max \{ dim(A) \mid A \mbox{ is an abelian subalgebra of } \mathcal{P}_L\},$$
$$\beta_L(\mathcal{P})= max \{ dim(I) \mid I \mbox{ is an abelian  ideal of } \mathcal{P}_L\}.$$
\end{defn}

Clearly, the numbers $\alpha$, $\alpha_A$ and $\alpha_L$ ($\beta$, $\beta_A$ and $\beta_L$ resp.) are actual invariants for a given Poisson algebra, as homomorphisms preserve abelian subalgebras (ideals resp.). Observe that $\alpha(\mathcal{P})\leq \alpha_A(\mathcal{P})$ and $\alpha(\mathcal{P}) \leq \alpha_L(\mathcal{P})$, and the same is true for $\beta$. Also, $\beta(\mathcal{P})\leq \alpha(\mathcal{P})$, $\beta_A(\mathcal{P})\leq \alpha_A(\mathcal{P})$ and $\beta_L(\mathcal{P})\leq \alpha_L(\mathcal{P})$. But there are no other obvious relations.
\begin{rem}
Observe that, for example, the Poisson algebras on the complex oscillator algebra ${\mathfrak P}_{1, \mu}^{1}$ satisfy that $\alpha_A({\mathfrak P}_{1, \mu}^{1})>\alpha_L({\mathfrak P}_{1, \mu}^{1})$ and $\beta_A({\mathfrak P}_{1, \mu}^{1})>\beta_L({\mathfrak P}_{1, \mu}^{1})$, see section \ref{osci}. Meanwhile, the nilpotent Poisson algebra $\mathcal{P}_{4,12}$ satisfies that $\alpha_A(\mathcal{P}_{4,12})<\alpha_L(\mathcal{P}_{4,12})$ and $\beta_A(\mathcal{P}_{4,12})<\beta_L(\mathcal{P}_{4,12})$, see Table~\ref{tab2}. Also, the Poisson algebras on the real oscillator algebra satisfy that $\alpha({\mathfrak P}_{1, \mu}^{1})<\beta({\mathfrak P}_{1, \mu}^{1})$.
\end{rem}

In the present section, we examine the relationship between these invariants. 

\subsection{Abelian subalgebras of codimension one}

Recall that a Poisson algebra with an abelian subalgebra of codimension one is 2-step solvable, as was noted in Remark \ref{nilpcodim1}.
Now, let us recall the situation with associative commutative algebras first.

\begin{prop}\label{ACCODIM1}
    Let $\mathcal{A}$ be an associative commutative algebra. Then $\alpha(\mathcal{A})=\beta(\mathcal{A})$. Moreover, if $A$ is an abelian subalgebra of dimension $\alpha(\mathcal{A})$, then it is an abelian ideal.
\end{prop}
\begin{proof}
    Suppose $A$ is an abelian subalgebra of dimension $\alpha(\mathcal{A})$. If it is not an ideal, there exist $a\in A$ and $x\in\mathcal{A}$ such that $ax\not\in A$. But then $A'=A+ \mathbb{F}(ax)$ is an abelian subalgebra of dimension $\alpha(\mathcal{A})+1$, which is a contradiction. So $A$ must be an ideal, and $\alpha(\mathcal{A})=\beta(\mathcal{A})$. 
    \end{proof}

This result does not hold for Lie algebras. However, it is possible to prove that if a Leibniz algebra has an abelian subalgebra of codimension one, then it has an abelian ideal of codimension one (see \cite{Towers} for more detail). For Poisson algebras, we can prove the following results.

\begin{lem}\label{lPACODIM1}
    Let $\mathcal{P}$ be a Poisson algebra of dimension $n$. If $\alpha(\mathcal{P})=n-1$, then $\beta(\mathcal{P})=n-1$.
\end{lem}
\begin{proof}
    Let $A$ be an abelian subalgebra of codimension one. Suppose $A = \textrm{span}(e_2, \ldots, e_n)$ and $\mathcal{P} = A \oplus \mathbb{F}e_1$ as a vector space. Then, the multiplication table of $\mathcal{P}$ is given by (for $1\leq i\leq n$)
    $$e_1 \cdot e_i = \sum_{k=1}^n \lambda_{ik}e_k, \quad \quad \quad \quad  [e_1, e_i] = \sum_{k=1}^n \mu_{ik}e_k.$$    
    
    Suppose that $A$ is not an ideal, otherwise the result is proved. By Proposition \ref{ACCODIM1}, $A$ is an ideal of $\mathcal{P}_A$, so assume $[e_1, A]\not\subseteq A$. Without loss of generality, suppose $[e_1, e_2]\not \in A$, that is, $\mu_{21}\neq0$. We will show that this implies that $\mathcal{P}_A$ is trivial, and consequently, $\mathcal{P}$ is a Lie algebra. In that case, the subspace $B = \textrm{span}([e_1, e_2], v_j:= \mu_{j1}\mu_{21}^{-1}e_2-e_j: 3\leq j \leq n)$ is an abelian ideal of codimension one as it was proved in \cite[Proposition 3.1]{BC12}. (Note that \cite[Proposition 3.1]{BC12} is stated for characteristic zero, but the proof actually works in any characteristic.) 

    \medskip
    
    \noindent\underline{Claim:} $e_1A = 0$ and $e_1e_1 = 0$.     
    Since $A$ is an ideal of $\mathcal{P}_A$, we have $0=[e_1e_i, e_2] = e_i[e_1, e_2]= \mu_{21}e_ie_1$ for $i>1$. It follows that $e_1A = 0$, because $\mu_{21}\neq0$. Moreover, using that $A$ is abelian, we have $$0=[e_1e_2,e_1]=[e_1,e_1]e_2+e_1[e_2,e_1]=-\mu_{21}e_1e_1.$$ Since $\mu_{21}\neq 0$ we have $e_1e_1=0$.

    \medskip
    
    Consequently, we have that $\mathcal{P}_A$ is trivial, and the result is proven.
\end{proof}

We have proven a more general result for non-trivial Poisson algebras.

\begin{thm}\label{PACODIM1}
Let $\mathcal{P}$ be a non-trivial Poisson algebra of dimension $n$ over an arbitrary field. Then an abelian subalgebra of codimension one is an abelian ideal.      
\end{thm}

\begin{rem}
    In the case of Poisson algebras with trivial $\mathcal{P}_A$, the theorem does not hold. For example, the solvable algebra with $[e_1, e_2] = e_2$ has an abelian subalgebra spanned by $e_1$ and it is not an ideal. Moreover, this nicety does not occur for abelian subalgebras of codimension two. For example, the algebra $\mathcal{P}_{3,20}$ in Table \ref{tab1} has an abelian subalgebra spanned by $e_3$ and it is not an ideal. In fact, its normalizer is the non-abelian subalgebra spanned by $e_1$ and $e_3$. 
\end{rem}

\subsection{Abelian subalgebras of codimension two} \label{secco2} First, we examine the case in which the field is arbitrary. After that, we consider the case in which the field is algebraically closed. Lastly, we weaken the restriction on the field but assume our Poisson algebra is nilpotent. Let us prove the following useful lemmas.

\begin{lem}\label{lPACODIM2}
    Let $\mathcal{P}$ be a Poisson algebra of dimension $n$ over {an arbitrary field} such that $\alpha(\mathcal{P})=n-2$. Let $A$ be an abelian subalgebra of codimension two such that $Q_a$ is not nilpotent for some $a\in A$ and $A$ is an ideal of a subalgebra $B$ of codimension one in $\mathcal{P}$. Then $\mathcal{P}$ has an abelian ideal of codimension two. 
\end{lem}
\begin{proof}
    
    Let $e_2, \ldots e_{n-1}$ be a basis of $A$ and $B=A\oplus\mathbb{F}e_1$.
    Suppose there is some $e_j\in A$ such that $Q_{e_j}$ is not nilpotent. 
        Then the Fitting decomposition of $\mathcal{P}$ with respect to the maps in $M(A)$ (see Theorem \ref{thmpa1}) is given by $\mathcal{P}=B\oplus \mathbb{F}x$ where $B$ is the null-space and 
        $\mathbb{F}x$ is the $1$-space with $x\in \mathcal{P}$. Hence, we can assume $P_{e_j}(x)=0$ and $Q_{e_j}(x)=\mu_j x$ for $j>1$ and $\mu_j\in \mathbb{F}$. Also, assume $\mu_2\neq0$. Then we claim that $A' = \textrm{span}(v_j: 3\leq j \leq n-1) \oplus \mathbb{F}x$ where $v_j=\mu_2^{-1}\mu_j e_2 - e_j$ is an abelian ideal of $\mathcal{P}$. 
        Clearly, it is an abelian subalgebra since      $v_j x = 0$, $[v_j, x] = \mu_2^{-1}\mu_j[e_2, x] - [e_j, x] = 0$ and        
         $xx= \mu_{2}^{-1} [e_2,x]x = \mu_2^{-1} ([xe_2, x] - e_2[x,x]) = 0$. Also, we have
         $$[e_1, x] = \mu_2^{-2}[e_1, [e_2, [e_2, x]]] = \mu_2^{-2} ([[e_1, e_2], [e_2, x]]+ [e_2, [[e_1, e_2], x]] + [e_2, [e_2, [e_1, x]]]) \in \mathbb{F}x.$$
        Moreover, we have $[e_1, v_j]=\mu_{2}^{-1}\mu_j[e_1, e_2]-[e_1, e_j]$ and $[e_1, e_j]\in A'$ for $j>1$. Otherwise, we obtain $A=\textrm{span}(v_j: 3\leq j \leq n-1) + \mathbb{F}[e_1, e_j]$, but $[[e_1, e_j], x]= [[e_1, x], e_j] + [e_1, [e_j, x]]=0$, implying that $Q_a(x)=0$ for all $a\in A$, which is a contradiction. Therefore, $[e_1, v_j]\in A'$ and $A'$ is an abelian ideal of codimension two of $\mathcal{P}$ by Proposition \ref{prop22}, proving the result.   
\end{proof}

\begin{thm}\label{lPACODIM2a}
    Let $\mathcal{P}$ be a Poisson algebra of dimension $n$ over {an arbitrary field} with $\alpha(\mathcal{P}) = n-2$. Let $A$ be an abelian subalgebra of codimension two. If there is a subalgebra of $\mathcal{P}$ of codimension one containing $A$, then one of the following situations occurs
    
    \begin{enumerate}
        \item If a subalgebra $B$ of $\mathcal{P}$ of codimension one containing $A$ is a Lie ideal. Then $\beta(\mathcal{P})= n-2$ or the Lie center of $B$ is an abelian ideal of dimension $n-3$, $B_A$ is zero and $\beta(\mathcal{P})= n-3$. Moreover, $\mathcal{P}_L$ is  $3$-step solvable.

        \item If there is a subalgebra $B$ of $\mathcal{P}$ of codimension one  containing $A$ which is not a Lie ideal, then $\beta(\mathcal{P})=n-2$ or
        $\mathcal{P}_L\cong L_1(\gamma)\oplus \mathbb{F}^{n-3}$, $\beta(\mathcal{P})=n-3$ and either $\mathcal{P}_A$ is zero or $\mathcal{P}$ is isomorphic to $\mathfrak{p}_n(\gamma)$.
    \end{enumerate}
\end{thm}
\begin{proof}
        Let $A$ be an abelian subalgebra of codimension two and let $B$ be a subalgebra of codimension one containing $A$. We can write $\mathcal{P}=B + \mathbb{F}x$ and $B=A + \mathbb{F}e_1$, as vector spaces. By Lemma~\ref{lPACODIM1}, we can assume that $A$ is an ideal of $B$. Thus, we have $B\subseteq N(A)$. If $N(A)=\mathcal{P}$, then the result is proven. So we suppose $N(A)=B$. Let $e_2, \ldots, e_{n-1}$ be a basis of $A$.
    Note that if $[x, A]\subset A$, then $A$ is an ideal of $\mathcal{P}$, by {Proposition~\ref{prop22}}. So assume
    $[x, e_2] \not \in A$. 
    Now, let us distinguish if $B$ is a Lie ideal of $\mathcal{P}_L$ or not.    

    \begin{enumerate}[leftmargin=*]
        \item \underline{If $B$ is an ideal of $\mathcal{P}_L$.} In this case, we can write $[x, e_2]=e_1$. Denote $[x, e_j] = \sum_{i=1}^{n-1}\alpha_{ji}e_i$. It follows that $BB=0$. Indeed, firstly we will show that $e_1e_2=0$. Consider $$xe_2=\beta_0 x+\beta_{21}e_1+\sum_{k=2}^{n-1}\beta_{2k}e_k.$$ 
        From the associativity we have $0=x(e_2e_2)=(xe_2)e_2$, which leads to $\beta_0xe_2+\beta_{21}e_1e_2=0$ and then $\beta_0xe_2=-\beta_{21}e_1e_2$. Since $A$ is an ideal of $B$ we have that $e_1e_2 \in A$ so $\beta_0=0$, as, otherwise, we have a contradiction. Thus we get $\beta_{21}e_1e_2=0$. If $\beta_{21}\neq 0$ the result follows, and if 
        $\beta_{21}= 0$, then $xe_2 \in A$ and so $0=[xe_2,e_2]=[x,e_2]e_2+x[e_2,e_2]=e_1e_2$. In any case $e_1e_2=0$. Now, for $j\geq1$ we obtain
        $$e_1e_j=[x, e_2]e_j=[x, e_2e_j]-[x, e_j]e_2= \alpha_{j1}e_1e_2=0.$$
        By Proposition \ref{ACCODIM1}, $B$ is an abelian ideal of $\mathcal{P}_A$. Also,  by \cite[Proposition 3.1]{Towers13}, we have $[B, B]=0$ or $\textrm{dim}([B, B])=1$ and the Lie center $C(B_L)$ of $B$ has codimension at most one in $A$. In the first case, $B$ is an abelian subalgebra of $\mathcal{P}$, which is a contradiction.         
        In the second case, we have $[B, B]= \mathbb{F}[e_1,e_2]$ and $[v_j, B]=0$ for $v_j=\alpha_{j1}e_2-e_j$ where 
        $3\leq j\leq n-1$. Observe that $Z:=C(B_L)$ is an ideal of $\mathcal{P}$, since we have 
        $$[xZ, B]\subseteq x[Z, B] + Z[x, B] = 0 \textrm{ and } [[x, Z], B]\subseteq[[x, B], Z] + [x, [Z, B]]=0.$$
        Therefore, $Z$ is an abelian ideal and we have $v_j\in Z$, so $\textrm{dim}(Z)\geq n-3$. Hence, $\beta(\mathcal{P})\geq n-3$. 
        {Lastly, observe that $\mathcal{P}_L$ is $3$-step solvable,  since $[\mathcal{P}, \mathcal{P}] = [B + \mathbb{F}x, B + \mathbb{F}x] \subseteq B$, $[B, B]\subseteq A$ and A is abelian.}

       \medskip
        \item  \underline{If $B$ is not an ideal of $\mathcal{P}_L$.} 
        By Lemma \ref{lPACODIM2}, we  can assume that $A$ acts nilpotently in $\mathcal{P}$. Then there is some $y\in\mathcal{P}$ such that $y\not\in B$ and $[y, A]+yA\subseteq B$. Thus, we have $\mathcal{P}=B+\mathbb{F}y$. Denote $[y, e_j] = \sum_{i=1}^{n-1}\alpha_{ji}e_i$ for $j>1$. Since $A$ is not an ideal, we can write  $[y, e_2]=e_1$ and assume $[y, e_j]\in A$ for $j>2$.         
        It follows that $[e_1, e_j]= [[y, e_2], e_j] = [[y, e_j], e_2] = 0$ for $j>2$. Also, we have $[y, e_1]=\lambda y + b$ where $\lambda\neq0$ and $b=\sum_{i=1}^{n-1} b_i e_i$, because $B$ is not a Lie ideal.       
        
        Now, suppose $[e_1, e_2]=0$, then $0=[e_2, [y, e_1]]=\lambda [e_2, y] + [e_2, b]$ and $\lambda[e_2, y]=0$, which is a contradiction. Hence, $[e_1, e_2]\neq0$. Therefore, $[e_1, e_2]$ is an eigenvector of $Q_{e_1}$, so write $[e_1, [e_1, e_2]]= \mu [e_1, e_2]$. Denote $[e_1, e_2]=\sum_{i=2}^{n-1} \gamma_i e_i$.
        
            \medskip
        
        \noindent\underline{Claim:} $\mu = \gamma_2 = \lambda \neq0$, $b_1 = 0$ and $[y, e_j] = 0$ for $j>2$. 
         Indeed, we have $$\mu [e_1, e_2] = [e_1, [e_1, e_2]] = \sum_{j=2}^{n-1} \gamma_j [e_1, e_j]= \gamma_2 [e_1, e_2].$$
         Therefore, we obtain $\mu=\gamma_2$.
         Also, since 
         \begin{equation}\label{eq13}
             [y, e_2] = \lambda^{-1} [[y, e_1], e_2]- \lambda^{-1}[b, e_2]= \lambda^{-1} \sum_{k=2}^{n-1}\gamma_k [y, e_k]- \lambda^{-1}b_1 [e_1, e_2],
         \end{equation}
         we obtain $\lambda=\gamma_2$.  Note that the eigenvalues of $Q_{e_1}$ restricted to $A$ are precisely $0$ and $\lambda\neq0$. But we also have for $j>2$ that
        $$[y, e_j]=\lambda^{-1}[[y, e_1], e_j] - \lambda^{-1}[b, e_j] = \lambda^{-1}[[y, e_j], e_1],$$
        implying that $[e_1, [y, e_j]] = - \lambda [y, e_j]$. Since $[y, e_j]\in A$, this implies that $[y, e_j] = 0$ for $j>2$. Moreover, we have $b_1 [e_1, e_2] = 0$ and $b_1=0$, by equation (\ref{eq13}).

            \medskip

         \noindent\underline{Claim:} {  Denote $e:= \lambda y+ b$,  $f:=[e_1, e_2]$ and $h:=e_1$. The subspace
        $\mathfrak{g}=\mathbb{F}e \oplus \mathbb{F}f \oplus \mathbb{F}h$ is a subalgebra of $\mathcal{P}_L$ isomorphic to $L_{1}(\Gamma)$. Note that $\mathfrak{g}$ is a subalgebra isomorphic to $L_{1}(\Gamma)$, since         
\begin{equation*}
\begin{split}
& [f, h] = [[e_1, e_2], e_1] = -\lambda f,\\
&[e, h] = [\lambda y + b, e_1] = \lambda^2 y + \lambda b - b_2[e_1, e_2] = \lambda e - b_2 f, \\
 &[f, e] = [[e_1, e_2], \lambda y + b] = -\lambda \sum_{k=2}^{n-1}\gamma_{k}[y, e_k] = -\lambda^2 e_1 = -\lambda^2 h.
\end{split}
\end{equation*}
        If we consider the basis $ E_1 = \ddfrac{1}{\lambda}h, E_2 = f, E_3 = \ddfrac{b_2}{\lambda^4}f -  \ddfrac{1}{\lambda^3}e$, we get the more familiar multiplication table $[E_1, E_2] = E_2, [E_1, E_3] = \gamma E_2 - E_3, [E_2, E_3] = E_1$ where $\gamma = \ddfrac{b_2}{\lambda^4}$, which is the simple Lie algebra $L_1(\Gamma)$ studied in \cite{amayo2}. }
        
        \medskip
        {
        \noindent\underline{Claim:} $\mathcal{P}_L\cong L_1(\Gamma)\oplus \mathbb{F}^{n-3}$ and $\mathcal{P}_A$ is zero or $\mathcal{P}\cong \mathfrak{p}_n(\gamma)$.
        Since          
        $\mathcal{P} = \mathfrak{g}\oplus \textrm{span}(e_j: 3\leq j\leq n-1)$ as Lie algebras, the first part is clear. Note that $[e, e_j] = [f, e_j] = [h, e_j]=0$ for $3\leq j\leq n-1$, because $[B, e_j] = 0$ and $[y, e_j]=0$. Consequently, the second part follows by Lemma \ref{sl2+}.    }

        \medskip

        Finally, observe that an abelian ideal of maximal dimension is $\textrm{span}(e_j: 3\leq j\leq n-1)$. An abelian ideal of dimension $n-2$ would imply that $L_1(\Gamma)$ has an ideal of at least dimension one or that $L_1(\Gamma)$ is abelian, which is not true.
    \end{enumerate}
    We have shown that there are two possibilities, and one excludes the other. The result is proven.
\end{proof}

{
\begin{lem}\label{maxLie}
    Let $\mathcal{P}$ be a Poisson algebra of dimension $n$ over {an arbitrary field}. Let
    $A$ be an abelian subalgebra of codimension two. Then $A$ is a maximal subalgebra of $\mathcal{P}$ if and only if $A$ is a maximal subalgebra of $\mathcal{P}_L$.
\end{lem}
\begin{proof}
    Suppose $A$ is not a maximal subalgebra of $\mathcal{P}_L$. Then there is a subalgebra $B$ of $\mathcal{P}_L$ of codimension one such that $B=A+\mathbb{F}b$. Suppose $ba\not\in A$ for some $a\in A$, then $A'=A+\mathbb{F}ba$ is a subalgebra of $\mathcal{P}$, which is a contradiction. Assume $bA\subset A$. The subalgebra generated by $b$ is nilpotent or contains an idempotent $e$. In the first case, there is some $k\in \mathbb{N}$ such that $b^k\not\in A$ and $b^{k+1}\in A$, then $A+\mathbb{F}b^k$ is a subalgebra of $\mathcal{P}$. In the second case, $A+\mathbb{F}e$ is a subalgebra of $\mathcal{P}$. Hence, $A$ is a maximal subalgebra of $\mathcal{P}_L$. The converse is clear.
\end{proof}}

Let us introduce the algebra $\mathfrak{q}_3(\lambda)$, where $\lambda = (\lambda_{ij}) \in M_2(\mathbb{F})$, with basis $h, x, y$ and {skew-symmetric} multiplication given by
$$[h, x] = \lambda_{11} x + \lambda_{12} y, \quad     [h, y] = \lambda_{21} x + \lambda_{22} y, \quad [x, y]=h.$$ 
 
 Likewise, we introduce the algebra $\mathfrak{q}_4(\lambda, \mu)$, where $\lambda = (\lambda_{ij})$, $\mu = (\mu_{ij}) \in M_2(\mathbb{F})$, with basis $h, a, x, y$ and {skew-symmetric} multiplication
$$[h, x] = \lambda_{11} x + \lambda_{12} y, \quad     [h, y] = \lambda_{21} x + \lambda_{22} y, \quad [a, x] = \mu_{11} x + \mu_{12} y, \quad     [a, y] = \mu_{21} x + \mu_{22} y, \quad [x, y]=h.$$

\begin{rem}\label{remq3}
    The algebra $\mathfrak{q}_3(\lambda)$ is a Lie algebra if $\textrm{trace}(\lambda)=0$, that is, if $\lambda\in \mathfrak{sl}_2(\mathbb{F})$.
\end{rem}

\begin{lem}\label{q3sim}
    If $\lambda\in \mathfrak{sl}_2(\mathbb{F})$ and $\textrm{det}(\lambda) \neq 0$, then $\mathfrak{q}_3(\lambda)$ is a simple Lie algebra. Moreover, if $\textrm{det}(\lambda) = 0$, then $\mathfrak{q}_3(\lambda)$ is solvable.
\end{lem}
\begin{proof}
    Note that since $\lambda\in \mathfrak{sl}_2(\mathbb{F})$, then $\lambda_{22}=-\lambda_{11}$. Let $I$ be an ideal of $\mathfrak{q}_3(\lambda)$. Then we have some $z= \chi_1 x + \chi_2 y + \chi_3 h \in I$. It follows that $[h, z] = \chi_1[h, x] + \chi_2 [h, y] \in I$ and $[x, [h, z]] = (\chi_1 \lambda_{12} - \chi_2 \lambda_{11}) h\in I$ and $[y, [h, z]] = -(\chi_1 \lambda_{11}  +\chi_2 \lambda_{21}) h \in I$. If $h\in I$ then $I=\mathfrak{g}$, so we have $(\chi_1 \lambda_{12} - \chi_2 \lambda_{11}) = (\chi_1 \lambda_{11} + \chi_2 \lambda_{21}) = 0$, implying $\chi_1\chi_2=0$. Assume $\chi_1\neq0$ and $\chi_2=0$ (the case $\chi_1=0$ and $\chi_2\neq 0$ is analogous). Then we have $[h, z] = \chi_1 [h, x] \in I$ and $[x, [h, x]] = \lambda_{12}  h$, but if $\lambda_{12}  = 0$, then $[h, x] = \lambda_{11} x \in I$ and $0\neq \lambda_{11} [x, y] = h \in I$. The second claim of the lemma is clear.
\end{proof}

\begin{rem}\label{reml1}
Moreover, if the characteristic polynomial of $\lambda$ has a root in $\mathbb{F}$, then $\mathfrak{q}_3(\lambda)\cong L_1(\gamma)$.
\end{rem}

Recall the annihilator of a Poisson algebra $\mathcal{P}$ is the ideal $\textrm{Ann}(\mathcal{P}) = \left\{x\in\mathcal{P}: [x, \mathcal{P}] = x\mathcal{P} =0 \right\}$.

\begin{thm}\label{lPACODIM2b}
    Let $\mathcal{P}$ be a Poisson algebra of dimension $n$ over {an arbitrary field} with $\alpha(\mathcal{P}) = n-2$. If
    $A$ is an abelian subalgebra of codimension two which is a maximal subalgebra. Then $\beta(\mathcal{P}) = n- 2$ or we have one of the following situations
    \begin{enumerate}
        \item $\mathcal{P}_L$ is $3$-step solvable and $\beta(\mathcal{P})\leq n-3$, see \cite[Theorem 3.5 (ii) or (iii)]{Towers13}.
        \item $\mathcal{P}_L = \mathfrak{q}_3(\lambda)\oplus \mathbb{F}^{n-3}$ and $\textrm{Ann}(\mathcal{P}) = C(\mathcal{P}_L) = \mathbb{F}^{n-3}$ is an abelian ideal of maximal dimension of $\mathcal{P}$, where $\lambda\in \mathfrak{sl}_2(\mathbb{F})$ and the characteristic polynomial of $\lambda$ is irreducible.
        \item $\mathcal{P}_L = \mathfrak{q}_4(\lambda, \mu) \oplus \mathbb{F}^{n-4}$ and $\textrm{Ann}(\mathcal{P}) 
 = C(\mathcal{P}_L) = \mathbb{F}^{n-4}$ is an abelian ideal of maximal dimension of $\mathcal{P}$, where $\lambda, \mu \in \textrm{span}(id, m)$, $id\in \mathfrak{sl}_2(\mathbb{F})$ denotes the identity matrix (assumes $\textrm{char}(\mathbb{F}) = 2$), with $m\in \mathfrak{sl}_2(\mathbb{F})$ and the characteristic polynomial of $m$ is irreducible.       
    \end{enumerate}
\end{thm}
\begin{proof}
    Let $A$ be an abelian subalgebra of codimension two which is a maximal subalgebra. Then $A$ is self-normalizing. By the Fitting decomposition with respect to $M(A)$ we can write $\mathcal{P}= A + \mathcal{P}_1$, where $\mathcal{P}_1 = \mathbb{F}x \oplus \mathbb{F}y$. Also, by Lemma~\ref{maxLie}, $A$ is a maximal subalgebra of $\mathcal{P}_L$. Note that, by the maximality of $A$, we have $C(\mathcal{P}_L)\subset A$.    
    
    First, assume $[x,  y] = 0$. If $a_1, \ldots, a_{n-2}$ is a basis of $A$, then the maps $ad_{a_1}, \ldots, ad_{a_{n-2}}$  can be realized as commuting maps in $\mathcal{P}_1$. Hence, $r:=\textrm{dim}(\textrm{span}(ad_{a_1}, \ldots, ad_{a_{n-2}}))\leq 2$, by a result by Schur, see \cite{maryam}. If $r=0$, the algebra $\mathcal{P}_L$ is abelian and $A$ is not maximal. If $r=1$, then $C(\mathcal{P}_L)$ has dimension $n-3$ and $C(\mathcal{P}_L)+P_1$ is an abelian subalgebra of dimension $n-1$, which is a contradiction. Indeed, assume $ad_{a_1}\neq0$ and write $ad_{a_i}=\epsilon_i ad_{a_1}$. Then $a_i-\epsilon_i a_1 \in C(\mathcal{P}_L)$ for $2\leq i\leq n-2$. {For $c\in C(\mathcal{P}_L)$, the map $P_{c}$ is nilpotent, so $P_{c}(x)= \tau_1 v$ and $P_{c}(y)= \tau_2 v$ for some $v\in \textrm{ker}(P_{c})\cap \mathcal{P}_1$. Since $A$ is a maximal subalgebra of $\mathcal{P}_L$, we have $a\in A$ such that $[a, v]\not\in A$. It follows that  $0 = a[c, x] = c[a, x] = [a, cx] = \tau_1 [a, v]$ and $0 = a[c, y] = c[a, y] = [a, cy] = \tau_2 [a, v]$. Therefore, $\tau_1 = \tau_2 = 0$ and $C(\mathcal{P}_L) = \textrm{Ann}(\mathcal{P})$.} Moreover, for $a\in A$, we have $x[x, a]= [x, ax] = 0$ and $x[y, a]= [y, ax] = 0$ (analogously, $y[x, a]= [x, ay] = 0$ and $y[y, a]= [y, ay] = 0$), but $x, y \in [A, x]+[A, y]$, otherwise $A$ is not maximal. Hence, $xx=xy=yy=0$, implying that $C(\mathcal{P}_L)+P_1$ is an abelian subalgebra (and ideal) of dimension $n-1$. 
    Therefore, $r=2$. By a similar argument, we have that $C(\mathcal{P}_L)$ has dimension $n-4$ and $C(\mathcal{P}_L)=\textrm{Ann}(\mathcal{P})$. Hence $C(\mathcal{P}_L)+P_1$ is an abelian ideal of codimension two.
    
    Now, assume $[x, y]\neq0$. We have one of the following two situations

        \medskip

    \noindent\underline{If $h:=[x, y]\in A$.} By Remark~\ref{remq3}, we have $[h, x] = \lambda_{11} x + \lambda_{12} y$, $[h, y] = \lambda_{21} x - \lambda_{11} y$. Denote $\lambda=(\lambda_{ij})\in \mathfrak{sl}_2$. It follows that $\mathfrak{g} = \mathbb{F}x+\mathbb{F}y+\mathbb{F}h$ is a Lie subalgebra of $\mathcal{P}_{L}$ and $\mathfrak{g}=\mathfrak{q}_3(\lambda)$. Moreover, it is a Lie ideal. 
    Now, distinguish three cases.

    \begin{enumerate}
        \item If $\textrm{rank}(\lambda)= 0$. Then $\mathcal{P}_L$ is 3-step solvable, since $\mathfrak{g}$ is 2-step nilpotent (in fact, $\mathfrak{g}$ is the Heisenberg algebra) and $$[\mathcal{P}, \mathcal{P}] = [A+\mathbb{F}x \oplus \mathbb{F}y, A+\mathbb{F}x \oplus \mathbb{F}y] \subseteq [A, x] + [A, y] + [x, y]\subseteq \mathbb{F}x + \mathbb{F}y + \mathbb{F}[x, y] = \mathfrak{g}.$$
        \item If $\textrm{rank}(\lambda)= 1$. Suppose $\mathbb{F}v = \mathbb{F}[h, x] + \mathbb{F}[h, y]$. Then $v$ is an eigenvector of $Q_h$, but $Q_h$ commutes with $M(A)$, so $\textrm{span}(v)$ is $M(A)$-invariant. This is a contradiction, because it implies that $\mathbb{F}v + A$ is a subalgebra containing $A$. Observe that since there is some $a\in A$ such that $Q_a(v) = [a, v]= \tau v\neq 0$, because $N(A)=A$, and the maps $P_{A}$ are nilpotent, then $vv = \tau^{-1}[a,v]v = \tau^{-1}[av, v] = 0$.

        \item If $\textrm{rank}(\lambda)= 2$. Then $\mathfrak{g}$ is simple, by Lemma~\ref{q3sim}.  
        Observe that if $B$ is an ideal of $\mathcal{P}$, then $\mathfrak{g}\subset B$ or $B \subset A$. If $B$ is also abelian, then $B\subset A$ and $[B, x], [B, y]\in B$ if and only if  $B \subset C(\mathcal{P}_L)$. 
 Denote $\mathcal{P}_L = \mathfrak{s}\oplus C(\mathcal{P}_L)$. It is clear that $\mathfrak{s}$ is semisimple. Let $e_1, \ldots, e_{r}, x, y$ be a basis of $\mathfrak{s}$. Assume $h=e_1$. For any $1\leq i\leq r$, we have a derivation $d_i\in \textrm{Der}(\mathfrak{g})$ such that $[e_i, x] = d_i(x)$, $[e_i, y] = d_i(y)$, $0=[e_i, h] = d_i(h) $ and $d_i(\mathfrak{g})\subseteq \mathbb{F}x+\mathbb{F}y$.  Note that if $d_1, \ldots, d_r$ are linearly dependent, then $C(\mathfrak{s})\neq0$. Moreover,       
        these derivations must commute, by the Jacobi identity. Furthermore, the trace of these matrices is zero, because $0 = [e_i, [x, y]] = [[e_i, x], y] + [x, [e_i, y]]$.
        Therefore, the maps $d_1, \ldots, d_r$ can be realized as linearly independent commuting elements of $\mathfrak{sl}_2$. 
        
        If the characteristic of $\mathbb{F}$ is $p\neq2$, then $\alpha(\mathfrak{sl}_2) = 1$ and it follows that $r\leq 1$. Consequently, we have $\beta_L(\mathcal{P})=n-3$ and we can write $\mathcal{P}_L=\mathfrak{g}\oplus \mathbb{F}^{n-3}$ as Lie algebras.  Using an argument used above, it follows that $\textrm{Ann}(\mathcal{P}) = C(\mathcal{P}_L)=\mathbb{F}^{n-3}$ is an abelian ideal of maximal dimension, and we have $\beta(\mathcal{P})=n-3$. 
        Moreover, suppose $\lambda$ has an eigenvalue, then there is some $v\not\in A$ such that $[h,v]\in \mathbb{F}v$ and $A+\mathbb{F}v$ is a Lie subalgebra, which is a contradiction. Therefore, the characteristic polynomial of $\lambda$ must be irreducible.

        If the characteristic of $\mathbb{F}$ is $p=2$, then $\alpha(\mathfrak{sl}_2) = 2$ and the abelian subalgebras of $\mathfrak{sl}_2$ of dimension two are spanned by the identity matrix ${id}$ and $m\in\mathfrak{sl}_2$. Assume $\mathfrak{s}$ is $4$-dimensional, otherwise we have an analogous situation to that above. Note that in this case, we have $\mathfrak{s}=\mathfrak{q}_4(\lambda, \mu)$, where $\lambda, \mu \in \textrm{span}(id, m)$ for some $m\in\mathfrak{sl}_2$.   
        If $m$ has an eigenvalue, then there is some $v$ such that $[a, v], [h,v]\in \mathbb{F}v$ and $A+\mathbb{F}v$ is a subalgebra, implying that $A$ is not maximal. Therefore, the characteristic polynomial of $m$ must be irreducible. If that is the case, we conclude that $\beta_L(\mathcal{P})= n-4$. As above, $\textrm{Ann}(\mathcal{P}) = C(\mathcal{P}_L)$ and $\beta(\mathcal{P})= n-4$.
    \end{enumerate}

    \medskip

    \noindent\underline{If $[x, y]\not\in A$.} Then we can construct a subalgebra strictly containing $A$, which is a contradiction. Indeed, we have
    $[[x, y], a_i]=[[x, a_i], y] + [x, [y, a_i]] \in \mathbb{F}[x, y]$ and if $[x, y]a_i \in A$, then we have that $\mathbb{F}[x, y] + A$ is a subalgebra containing $A$. Write $[a_i, [x, y]] = \tau_i [x, y]$. Since $N(A)=A$, we can assume $\tau_3\neq0$ and it follows $[x, y][x, y] = \tau_3^{-1}[a_3, [x, y]][x, y]= \tau_3^{-1}[a_3[x, y], [x, y]] \in  \mathbb{F}[x, y]$. 
    If we assume $[x, y]a_3 \not\in A$, then 
    $\mathbb{F}[x, y]a_3+A$ is a subalgebra, since $[x, y]a_3 a_j = 0$ and $[[x, y]a_3, a_j] = [[x, y], a_j] a_3 = [[x, a_j], y] a_3 + [x, [y, a_j]] a_3\in \mathbb{F}[x, y]a_3$.  

\medskip

    Finally, note that, by \cite[Theorem 3.5 (ii) or (iii)]{Towers13}, the case in which $\mathcal{P}_L$ is solvable implies that $\beta(\mathcal{P}_L)\leq n-3$, so $\beta(\mathcal{P})\leq n-3$. 
\end{proof}

\begin{rem}
    In case $(3)$ of Theorem~\ref{lPACODIM2b}, the characteristic polynomial of $m=(m_{ij})\in \mathfrak{sl}_2$ is $p(t)= t^2 + m_{11}^2+m_{12}m_{21}$. There is no loss in generality in assuming $m_{11}=0$ and $m_{12}=1$. Denote $m_{21} = \rho$. Hence, the characteristic polynomial $p(t) = t^2 + \rho$ must be irreducible. If the field is finite with order $2^m$, we have $t=\rho^{2^{m-1}}$ a root of $t^2+\rho$. Therefore, case (3) is not possible over finite fields. Consider the field $\mathbb{F}= \mathbb{Z}/2\mathbb{Z}(s)$, that is the rational functions in the indeterminate $s$ in $\mathbb{Z}/2\mathbb{Z}$. Then the algebra $L$ with basis $x, y, h, a$ and multiplication
$$[h, x] = x, \quad [h, y] = y, \quad  [a, x] = sy, \quad [a, y]= x, \quad [x, y] = h.$$
is an example of a Lie algebra with an abelian subalgebra of codimension two $A = \textrm{span}(a, h)$ which is a maximal subalgebra, while $L$ is semisimple, $\alpha(L)=2$ and $\beta(L)=0$. Note that in this case $t^2+ s$ has no root in $\mathbb{F}$. 
\end{rem}

{  Combining Theorem~\ref{lPACODIM2a} and Theorem~\ref{lPACODIM2b}, we have the next general result. See also Remark~\ref{reml1}.

\begin{thm}
        Let $\mathcal{P}$ be a Poisson algebra of dimension $n$ over {an arbitrary field} with $\alpha(\mathcal{P}) = n-2$. Then $\beta(\mathcal{P}) = n- 2$ or we have one of the following situations

        \begin{enumerate}
        \item $\mathcal{P}_L$ is $3$-step solvable and $\beta(\mathcal{P})\leq n-3$.
        \item $\mathcal{P}_L = \mathfrak{q}_3(\lambda)\oplus \mathbb{F}^{n-3}$ and $\textrm{Ann}(\mathcal{P}) = C(\mathcal{P}_L) = \mathbb{F}^{n-3}$ is an abelian ideal of maximal dimension of $\mathcal{P}$, where $\lambda\in \mathfrak{sl}_2(\mathbb{F})$  and $det(\lambda)\neq 0$.
        \item $\mathcal{P}_L = \mathfrak{q}_4(\lambda, \mu) \oplus \mathbb{F}^{n-4}$ and $\textrm{Ann}(\mathcal{P}) 
 = C(\mathcal{P}_L) = \mathbb{F}^{n-4}$ is an abelian ideal of maximal dimension of $\mathcal{P}$, where $\lambda, \mu \in \textrm{span}(id, m)$, $id\in \mathfrak{sl}_2(\mathbb{F})$ denotes the identity matrix (assumes $\textrm{char}(\mathbb{F}) = 2$), with $m\in \mathfrak{sl}_2(\mathbb{F})$ and the characteristic polynomial of $m$ is irreducible.   
        \end{enumerate}

\end{thm}

In particular, for Lie algebras we have proved the next result that generalizes \cite[Theorem 3.5]{Towers13}.

\begin{cor}
    Let $L$ be a Lie algebra of dimension $n$ over {an arbitrary field} with $\alpha(L) = n-2$. Then $\beta(L) = n-2$ or we have one of the following situations
    \begin{enumerate}
        \item $L$ is $3$-step solvable and $\beta(L)\leq n-3$. This case corresponds to \cite[Theorem 3.5 (ii) and (iii)]{Towers13}.
        \item $L = \mathfrak{q}_3(\lambda)\oplus \mathbb{F}^{n-3}$ and $C(L) = \mathbb{F}^{n-3}$ is an abelian ideal of maximal dimension, where $\lambda\in \mathfrak{sl}_2(\mathbb{F})$ and $det(\lambda)\neq 0$.
        \item $L=\mathfrak{q}_4(\lambda, \mu) \oplus \mathbb{F}^{n-4}$ and $C(L) = \mathbb{F}^{n-4}$ is an abelian ideal of maximal dimension, where $\lambda, \mu \in \textrm{span}(id, m)$, $id\in \mathfrak{sl}_2(\mathbb{F})$ denotes the identity matrix (assumes $\textrm{char}(\mathbb{F}) = 2$), with $m\in \mathfrak{sl}_2(\mathbb{F})$ and the characteristic polynomial of $m$ is irreducible.  
    \end{enumerate}
\end{cor}}

\begin{thm}\label{PACODIM2}
    Let $\mathcal{P}$ be a Poisson algebra of dimension $n$ over an algebraically closed field. If $\alpha(\mathcal{P})=n-2$, then
 $\beta(\mathcal{P})=n-2$ or
        $\mathcal{P}_L\cong L_1(\gamma)\oplus \mathbb{F}^{n-3}$, $\beta(\mathcal{P})=n-3$ and either $\mathcal{P}_A$ is zero or $\mathcal{P}$ is isomorphic to $\mathfrak{p}_n(\gamma)$.
\end{thm}
\begin{proof}
    Let $A$ be an abelian subalgebra of codimension two. By Theorem \ref{thmpa1}, this subalgebra is not maximal. Therefore, there is a subalgebra $B$ of codimension one containing $A$. We can assume $A$ is an ideal of $B$. By Theorem \ref{lPACODIM2a}, we have two possibilities. The case $(2)$ gives us the claim in the statement, so assume we are in case $(1)$. Denote by $Z$ the Lie center of $B$. If $Z$ has dimension $n-2$, it is an abelian ideal of codimension two in $\mathcal{P}$ and the result is proved. So assume $\textrm{dim}(Z)=n-3$ and $Z=\textrm{span}(v_j: 3\leq j \leq n-1)$, following the notation of Theorem~\ref{lPACODIM2a}. 
    
    Suppose $[e_1, e_2]\not\in Z$, then $Z + \mathbb{F}[e_1, e_2]= A$ and we have $x[e_1, e_2]=[xe_1, e_2]-e_1[x, e_2]\in \mathbb{F}[e_1, e_2]$ and $[x, [e_1, e_2]]= [[x, e_1], e_2]+ [e_1, [x, e_2]] = \alpha_{11}[e_1, e_2]\in \mathbb{F}[e_1, e_2]$. Hence, $A$ is an ideal which is a contradiction.      
    Now, suppose $[e_1, e_2]\in Z$.        
        The spaces $B$ and $Z$ are $Q_x$-invariant, so there is some $b\in B$ such that $Q_x(b)=\lambda b + z$ with $z\in Z$ and $b\not\in Z$. We claim that $A':=Z + \mathbb{F}b$ is an abelian ideal of codimension two. Clearly, it is an abelian subalgebra since $A'A'=0$ and $[A', A'] = 0$. Also, we have $BA'\subseteq BB=0$ and $[B, A']\subseteq [B, B] \subseteq \mathbb{F}[e_1, e_2]\subseteq A'$. Moreover, we have $[x, b] = \lambda b + z \in A'$. By Proposition \ref{prop22}, $A'$ is an ideal of $\mathcal{P}$, concluding that $\beta(\mathcal{P})=n-2$.      
\end{proof}

If the field is not algebraically closed, the result does not hold, see Remark \ref{realoscillator}. Moreover, we have proven the following result for Poisson algebras with solvable Lie part.

\begin{cor}
     Let $\mathcal{P}$ be a Poisson algebra of dimension $n$ over an algebraically closed field. If $\alpha(\mathcal{P})=n-2$ and $\mathcal{P}_L$ is solvable, then $\beta(\mathcal{P})=n-2$. Particularly, the result holds for solvable Lie algebras.
\end{cor}

The restriction on the field cannot be freely removed, see \cite[Examples 3.1 and 3.2]{Towers13}. 
Although, it can be dropped if we assume that $\mathcal{P}$ is nilpotent.

\begin{thm}
Let $\mathcal{P}$ be a nilpotent Poisson algebra of dimension $n$. If $\alpha(\mathcal{P})=n-2$, then $\beta(\mathcal{P})=n-2$. Moreover, if $\alpha_A(\mathcal{P})=n-2$, then any abelian subalgebra of codimension two is an ideal.
\end{thm}
\begin{proof}
Let $A$ be an abelian subalgebra of codimension two. Assume $A$ is not a ideal of $\mathcal{P}$. By Proposition \ref{normnil}, $B:=N(A)$ is an ideal of codimension one containing $A$. Therefore, we can write $\mathcal{P}=B + \mathbb{F}x$ and $B=A + \mathbb{F}e_1$ as vector spaces. Let $e_2, \ldots, e_{n-1}$ be a basis of $A$.  By {Proposition~\ref{prop22}}, we assume
    $[x, e_2] \not \in A$. Since $[x, e_2]\in B$,  without loss of generality we can write $[x, e_2]=e_1$. 
    Denote $[x, e_j] = \sum_{i=1}^{n-1}\alpha_{ji}e_i$.     
    
    { By a similar argument to that of Theorem \ref{PACODIM1}, we have $BB=0$. 
        By \cite[Proposition 3.1]{Towers13}, we have $[B, B]=0$ or $\textrm{dim}([B, B])=1$ and the Lie center $C(B_L)$ of $B$ has codimension at most one in $A$. In the first case, $B$ is an abelian subalgebra of $\mathcal{P}$, which is a contradiction.         
        In the second case, we have $[B, B]= \mathbb{F}[e_1,e_2]$ and $[v_j, B]=0$ for $v_j=\alpha_{j1}e_2-e_j$ where 
        $3\leq j\leq n-1$. Observe that $Z:=C(B_L)$ is an ideal of $\mathcal{P}$, since we have 
        $$[xZ, B]\subseteq x[Z, B] + Z[x, B] = 0 \textrm{ and } [[x, Z], B]\subseteq[[x, B], Z] + [x, [Z, B]]=0.$$
        Consequently, if $Z$ has dimension $n-2$, it is an abelian ideal of codimension two in $\mathcal{P}$ and the result is proved. So assume $\textrm{dim}(Z)=n-3$ and $Z=\textrm{span}(v_j: 3\leq j \leq n-1)$.
        Suppose $[e_1, e_2]\not\in Z$, then $Z + \mathbb{F}[e_1, e_2]= A$ and we have $x[e_1, e_2]=[xe_1, e_2]-e_1[x, e_2]\in \mathbb{F}[e_1, e_2]$ and $[x, [e_1, e_2]]= [[x, e_1], e_2]+ [e_1, [x, e_2]] = \alpha_{11}[e_1, e_2]\in \mathbb{F}[e_1, e_2]$. Hence, $A$ is an ideal which is a contradiction.          
        Next, suppose $[e_1, e_2]\in Z$. }

        Since $\mathcal{P}$ is nilpotent, $\mathcal{P}_L$ is nilpotent. Therefore, the map $Q_x$ is nilpotent. Hence, there is some $k\geq1$ such that $Q_x^{k}(e_2)\not\in Z$, but $Q_x^{k+1}(e_2)\in Z$. Set $A'=Z+\mathbb{F}Q_x^{k}(e_2)$. We claim $A'$ is an abelian ideal of codimension two in $\mathcal{P}$.
        Observe that it is an abelian subalgebra, since $A'A'=0$ and $[A', A'] = 0$. Also, we have $BA'\subseteq BB=0$ and $[B, A']\subseteq [B, B] \subseteq \mathbb{F}[e_1, e_2]\subseteq A'$. Moreover, we have $[x, Q_x^{k}(e_2)] \in Z\subseteq A'$. By Proposition \ref{prop22}, $A'$ is an ideal of $\mathcal{P}$. 
        
        Finally, observe that if we additionally assume  $\alpha_A(\mathcal{P})=n-2$, then $BB=0$ is a contradiction, so any abelian subalgebra $A$ must be an abelian ideal of $\mathcal{P}$.
\end{proof}

\section{Invariants $\alpha$ and $\beta$ for distinguished Poisson  algebras}

This section is devoted to the study of the invariants $\alpha$ and $\beta$ for some important families of finite-dimensional Poisson  algebras. 

\subsection{Poisson algebras on the oscillator algebra}
\label{osci}

Recall the definition of the oscillator algebra.

\begin{defn}
Let $\lambda = (\lambda_1, \ldots, \lambda_{n}) \in {\mathbb R}^n$ be such that $0< \lambda_1 \leq \ldots \leq \lambda_n$. The oscillator algebra $\mathfrak{g}_{\lambda}^{n}$  is the 
real (or complex) vector space  with basis $e_{-1}, e_{0}, e_{i}, \hat{e}_{i}$ where $1\leq i\leq n$ together with the bracket given for $1\leq i\leq n$ by
$$[e_{-1}, e_{i}] = \lambda_{i} \hat{e}_{i}, \quad [e_{-1}, \hat{e}_{i}] = - \lambda_{i}e_i,  \quad [e_{i}, \hat{e}_{i}] = e_{0}.$$
The classical definition fixes $n=1$ and $\lambda_1 = 1$.
\end{defn}

The oscillator algebra is an example of a solvable, but not supersolvable algebra over $\mathbb{R}$. Recall that a Lie algebra $\mathcal{L}$ is supersolvable if there is a chain $0=\mathcal{L}_0\subset \mathcal{L}_1\subset \ldots \subset \mathcal{L}_n=\mathcal{L}$ where $\mathcal{L}_i$ is an ideal of dimension $i$ of $\mathcal{L}$. Any supersolvable Lie algebra is also solvable. By Lie’s theorem, these classes coincide over an
algebraically closed field of characteristic zero. 

The Poisson structures with underlying Lie algebra an oscillator algebra have been studied in \cite{oscillator}. It was shown that those Poisson algebras are precisely $(\mathfrak{g}_{\lambda}^{n}, \circ, [\cdot, \cdot])$ where $(\mathfrak{g}_{\lambda}^{n}, \circ)$ is given by $e_{-1}\circ e_{-1} = \mu e_0$ for some $\mu \in {\mathbb R}$ and all other products equal to zero.
Denote by ${\mathfrak P}_{\lambda, \mu}^{n}$ this family of Poisson algebras.

\begin{rem} \label{realoscillator}
    Let $\lambda = (\lambda_1, \ldots, \lambda_{n}) \in {\mathbb R}^n$ such that $0< \lambda_1 \leq \ldots \leq \lambda_n$ and $\mu \in \mathbb{R}$. If ${\mathfrak P}_{\lambda, \mu}^n$ is a real algebra, then we have 
    $\alpha({\mathfrak P}_{\lambda, \mu}^n)=n+1$ and $\beta({\mathfrak P}_{\lambda, \mu}^n)=1$. An abelian subalgebra of maximal dimension is $A=\textrm{span}(e_0, e_1, \ldots, e_n)$ and the abelian ideal of dimension one is $B=\mathbb{R}e_0$. The first claim can be proved with similar arguments to those used in the next theorem. For the second claim, suppose $B'$ is a bigger abelian ideal and let $x\in B'$ with $x=\sum \alpha_i e_i + \sum \beta_i \hat{e}_i$. Then $[x, [e_{-1}, x]]=0$ implies that $\sum (\lambda_i\alpha_i^2 + \lambda_i\beta_i^2) = 0$, which has no solution for non all zero $\alpha_i, \beta_i\in\mathbb{R}$.
\end{rem}

On the other hand, if ${\mathfrak P}_{\lambda, \mu}^n$ is a complex algebra, the situation is the following.

\begin{thm}
    Let $\lambda = (\lambda_1, \ldots, \lambda_{n}) \in {\mathbb R}^n$ such that $0< \lambda_1 \leq \ldots \leq \lambda_n$ and $\mu \in \mathbb{R}$. If ${\mathfrak P}_{\lambda, \mu}^n$ is a complex algebra, then we have 
    $\alpha({\mathfrak P}_{\lambda, \mu}^n)=\beta({\mathfrak P}_{\lambda, \mu}^n)=n+1$.    
\end{thm}
\begin{proof}
First, we claim that the subspace $A=\textrm{span}(e_0, e_1+ {\bf i}\hat{e}_1, \ldots, e_n+{\bf i}\hat{e}_n)$ is an abelian ideal of dimension $n+1$. It is clear that it is an abelian subalgebra. Also, we have 
$[e_{i}, e_i+ {\bf i}\hat{e}_i] = {\bf i} e_0 $, $[\hat{e}_{i}, e_i+ {\bf i}\hat{e}_i] = e_0$ and $[e_{-1}, e_i+ {\bf i}\hat{e}_i] = \lambda (\hat{e}_i- {\bf i}e_i) \in \mathbb{F}(e_i+ {\bf i}\hat{e}_i)$. The rest of the products are zero, so $A$ is an ideal. 

Now, let $A'$ be an abelian subalgebra and assume $\textrm{dim}(A')=m+1>n+1$. Note that $e_0\in A'$ and $e_{-1}$ is not in the support of $A'$. Also, for $1\leq i\leq n$,  either $e_i$ or $\hat{e}_i$ is in the support of $A'$. 
Observe that we can choose a basis $e_0, x_1, \ldots, x_m$ of $A'$ such that the pivot element of $x_i$ is $e_i$ or $\hat{e}_i$ for $i\leq n$. 
Assume we applied gaussian reduction with respect to these pivots and renamed the basis elements.
Let the pivot element of $x_{n+1}$ be $e_t$ (resp.  $\hat{e}_t$), so the pivot of $x_t$ is $\hat{e}_t$ (resp.  ${e}_t$). Then since $[x_{n+1}, x_t]=0$, there is some $k$ such that $e_k$ and $\hat{e}_{k}$ are in the support of $x_{n+1}$ or $x_t$, which is a contradiction since $x_k$ has pivot element $e_k$ or $\hat{e}_{k}$. Hence, $\alpha({\mathfrak P}_{\lambda, \mu}^n)=n+1$.
\end{proof}


\subsection{Poisson algebras on the null-filiform and filiform associative commutative algebras} The complex Poisson structures with underlying associative commutative algebra a null-filiform or filiform algebra were studied in \cite[Theorem 3.2 and 3.4]{pa3}. Any such Poisson algebra $\left( \mathcal{P},\circ , [\cdot, \cdot]\right)$ is isomorphic to one of the following algebras with basis $e_{1},e_{2},\dots ,e_{n}$ and $2\leq i+j\leq n-1$ unless indicated otherwise.
\begin{tasks}(2)
\task $\mathcal{P}_{0}^{n}:\quad e_{i}\circ e_{j}=e_{i+j},$  with $2\leq i+j\leq n$.

\task $\mathcal{P}_{1,1}^{n}:e_{i}\circ
e_{j}=e_{i+j}$.

\task $\mathcal{P}_{1,2}^{n}:\left\{ 
\begin{array}{c}
e_{i}\circ e_{j}=e_{i+j}, \\ 
\left[e_{1},e_{n}\right] = e_{n}.
\end{array}%
\right. $ 

\task $\mathcal{P}_{1,3}^{n}:\left\{ 
\begin{array}{c}
e_{i}\circ e_{j}=e_{i+j}, \\ 
\left[e_{1},e_{n}\right] =e_{n-1}.
\end{array}%
\right. $ 

\task $\mathcal{P}_{1,4}^{n}:e_{i}\circ e_{j}=e_{i+j},e_{n}\circ
e_{n}=e_{n-1}.$

\task $\mathcal{P}_{1,5}^{n}:\left\{ 
\begin{array}{c}
e_{i}\circ e_{j}=e_{i+j},e_{n}\circ e_{n}=e_{n-1}, \\ \left[e_{1},e_{n}\right] =e_{n-1}.
\end{array}%
\right.$
\end{tasks}

\begin{rem}
By Proposition \ref{ACCODIM1}, it is an straightforward verification that $\alpha(\mathcal{P}_{0}^{n}) = \beta(\mathcal{P}_{0}^{n})= \lceil n/2 \rceil$, $\alpha(\mathcal{P}_{1,1}^{n}) = \beta(\mathcal{P}_{1,1}^{n})= \lceil (n+1)/2 \rceil$ and $\alpha(\mathcal{P}_{1,4}^{n}) = \beta(\mathcal{P}_{1,4}^{n})= \lceil n/2 \rceil$. Moreover, using that the other algebras can be seen as deformations of the previous ones, we have $\alpha(\mathcal{P}_{1,2}^{n}) = \beta(\mathcal{P}_{1,2}^{n})=\lceil (n+1)/2\rceil$, $\alpha(\mathcal{P}_{1,3}^{n}) = \beta(\mathcal{P}_{1,3}^{n})=\lceil (n+1)/2 \rceil$ and $\alpha(\mathcal{P}_{1,5}^{n}) = \beta(\mathcal{P}_{1,5}^{n})= \lceil n/2 \rceil$.
\end{rem}

\subsection{Poisson algebras on the model filiform Lie algebra}

 The complex Poisson structures on the filiform Lie algebras have not been studied previously. Any filiform Lie algebra is a deformation of the model filiform Lie algebra $L^{n}$ given by the complex space with basis $x_0,x_1,x_2, \dots x_{n-1}$ and the non-trivial products $[x_0,x_{i}]=x_{i+1}$, where $\ 1 \leq i \leq n-2$, see \cite{Vergne}.  
 Let us study these Poisson structures.

\begin{thm} \label{n-filiform}Let $\mathcal{P}$ be a complex Poisson algebra such that $\mathcal{P}_L$ is isomorphic to the model filiform Lie algebra of dimension $n$, $L^n$ with $n \geq 3$.  Then for some $\lambda_1, \lambda_2, \lambda_3 \in {\mathbb C}$, $\mathcal{P}_A$ is totally determined by the only non-null commutative and associative products  following:
$$\left\{ \begin{array}{cc}
x_0x_0= &\lambda_1 x_{n-1},  \\
x_0x_1= &\lambda_2 x_{n-1},  \\
x_1x_1= &\lambda_3 x_{n-1}.  \\
\end{array}\right. $$
    
\end{thm}

\begin{proof} Consider the basis $x_0,x_1,x_2, \dots, x_{n-1}$ such that $L^n$ is expressed by $ [x_0,x_{i}]=x_{i+1}$ with  $1 \leq i \leq n-2$. Let us denote now the associative product as following
$$x_ix_j=x_jx_i=\sum_{k=0}^{n-1}\lambda_{ij}^k x_k, \quad  1\leq i \leq j \leq n-1.$$
Firstly, by Leibniz rule $[x_0x_0,x_0]=2x_0[x_0,x_0]=0$ we have that $x_i \notin x_0x_0$ for all $i$, $1\leq i \leq n-2$, obtaining then  ${ x_0x_0=\lambda_{00}^0 x_0+ \lambda_{00}^{n-1} x_{n-1}}$. Similarly, from $[x_1x_1,x_1]=0$ we get $x_0\notin x_1x_1$, so $x_1x_1=\sum_{k=1}^{n-1}\lambda_{11}^k x_k$. As $[x_1x_1,x_0]=2x_1[x_1,x_0]=-2x_1x_2$, then $$x_1x_2=\frac{1}{2}[x_0,\sum_{k=1}^{n-1}\lambda_{11}^k x_k]=\frac{1}{2}(\sum_{k=1}^{n-2}\lambda_{11}^{k} x_{k+1}).$$
Now, since $[x_0x_1,x_1]=[\lambda_{01}^0 x_0,x_1]=[x_0,x_1]x_1=x_1x_2$, then ${ x_1x_2=\lambda_{01}^0 x_2}$ and consequently $\lambda_{11}^k=0$ for $2\leq k\leq n-2$ and $\lambda_{01}^0=\frac{1}{2}\lambda_{11}^1$, remaining then ${ x_1x_1=2\lambda_{01}^0 x_1+\lambda_{11}^{n-1} x_{n-1}}$.
Next, consider $[x_0x_1,x_0]=x_0[x_1,x_0]=-x_0x_2$, therefore $x_0x_2=[x_0,\sum_{k=0}^{n-1}\lambda_{01}^k x_k]=\sum_{k=1}^{n-2}\lambda_{01}^k x_{k+1}$. But from $[x_0x_0,x_1]=\lambda_{00}^0x_2=2x_0[x_0,x_1]=2x_0x_2$ we get ${ x_0x_2=\frac{1}{2}\lambda_{00}^0 x_2}$, and consequently $\lambda_{01}^k=0$ for $2\leq k\leq n-2$ and $\lambda_{01}^1=\frac{1}{2}\lambda_{00}^0$, remaining then ${ x_0x_1=\lambda_{01}^0 x_0+\frac{1}{2}\lambda_{00}^0 x_1+\lambda_{01}^{n-1} x_{n-1}}$. Likewise, since $[x_0x_0,x_{i-1}]=\lambda_{00}^0x_i=2x_0[x_0,x_{i-1}]=2x_0x_i$ we get ${ x_0x_i=\frac{1}{2}\lambda_{00}^0 x_i}$,  $ 3\leq i \leq n-1$. Finally, from $[x_0x_1,x_{i-1}]=[\lambda_{01}^0 x_0,x_{i-1}]=[x_0,x_{i-1}]x_1=x_1x_i$, we have ${ x_1x_i=\lambda_{01}^0 x_i}$, $ 3\leq i \leq n-1$. Moreover, $0=[x_0x_2,x_1]=x_2x_2$ and for $2\leq i,j\leq n-1$ since $0=[x_0x_i,x_j]=[x_0,x_j]x_i=x_ix_{j+1}$. Thus, the only non-null associative products at this point are the following

$$\left\{ \begin{array}{cll}
x_0x_0= &\lambda_{00}^0 x_0+ \lambda_{00}^{n-1} x_{n-1} & \\
x_0x_1= &\lambda_{01}^0 x_0+\frac{1}{2}\lambda_{00}^0 x_1+\lambda_{01}^{n-1} x_{n-1}  &\\
x_0x_2=&\frac{1}{2}\lambda_{00}^0 x_2&\\
x_0x_i=&\frac{1}{2}\lambda_{00}^0 x_i,  & 3\leq i \leq n-1\\
x_1x_1=&2\lambda_{01}^0 x_1+\lambda_{11}^{n-1} x_{n-1} & \\
x_1x_2=&\lambda_{01}^0 x_2&\\
x_1x_i=&\lambda_{01}^0 x_i, &3\leq i \leq n-1
\end{array}\right. $$

Secondly, from the associativity of $(x_0x_0)x_1=x_0(x_0x_1)$ it is obtained that $(\lambda_{00}^0 x_0+ \lambda_{00}^{n-1} x_{n-1})x_1=x_0(\lambda_{01}^0 x_0+\frac{1}{2}\lambda_{00}^0 x_1+\lambda_{01}^{n-1} x_{n-1} )$ which leads to 
$\lambda_{00}^0 x_0x_1+ \lambda_{00}^{n-1} x_{n-1}x_1=\lambda_{01}^0 x_0x_0+\frac{1}{2}\lambda_{00}^0 x_0x_1+\lambda_{01}^{n-1} x_0x_{n-1} $ or equivalently to $$\frac{1}{2}\lambda_{00}^0 x_0x_1+ \lambda_{00}^{n-1} x_{n-1}x_1=\lambda_{01}^0 x_0x_0+\lambda_{01}^{n-1} x_0x_{n-1} $$
$$\frac{1}{2}\lambda_{00}^0 (\lambda_{01}^0 x_0+\frac{1}{2}\lambda_{00}^0 x_1+\lambda_{01}^{n-1} x_{n-1})+ \lambda_{00}^{n-1} \lambda_{01}^0 x_{n-1}=\lambda_{01}^0 (\lambda_{00}^0 x_0+ \lambda_{00}^{n-1} x_{n-1} )+\lambda_{01}^{n-1} \frac{1}{2}\lambda_{00}^0 x_{n-1} $$
therefore $\frac{1}{2}(\lambda_{00}^0)^2x_1=0$, so $\lambda_{00}^0=0$. Next, by the associativity of $x_0(x_1x_1)=(x_0x_1)x_1$ it is obtained that $x_0(2\lambda_{01}^0 x_1+\lambda_{11}^{n-1} x_{n-1})=(\lambda_{01}^0 x_0+\lambda_{01}^{n-1} x_{n-1})x_1$ which leads to 
$$2\lambda_{01}^0 x_0x_1=\lambda_{01}^0 x_0x_1+\lambda_{01}^{n-1} x_1x_{n-1}$$
$$\lambda_{01}^0 (\lambda_{01}^0 x_0+\lambda_{01}^{n-1} x_{n-1})=\lambda_{01}^{n-1} \lambda_{01}^0 x_{n-1}$$
therefore $(\lambda_{01}^0)^2x_0=0$, so $\lambda_{01}^0=0$, obtaining the expression of the statement after renaming $\lambda_{00}^{n-1}=\lambda_1$, $\lambda_{01}^{n-1}=\lambda_2$ and $\lambda_{11}^{n-1}=\lambda_3$.       
\end{proof}

\begin{rem}
    Let $\mathcal{P}$ be under the conditions of Theorem \ref{n-filiform}, then $\alpha_L=\beta_L=n-1$ and is given by $\textrm{span}(x_1,\dots,x_{n-1})$. We distinguish four cases:
    \begin{enumerate}
        \item $\lambda_1 \lambda_3 \neq 0$. In this case $\alpha_A(\mathcal{P})=\beta_A(\mathcal{P})=n-2$ and is given by $\textrm{span}( x_2,\dots,x_{n-1})$. Therefore, $\alpha(\mathcal{P})=\beta(\mathcal{P})=n-2$. 

        \item $\lambda_1= \lambda_3 = 0$. In this case $\lambda_2\neq 0$ or $\mathcal{P}_A$ is trivial, then   $\alpha_A(\mathcal{P})=\beta_A(\mathcal{P})=n-1$ and is given for example by $\textrm{span}(x_1,\dots,x_{n-1})$. Also it can be considered $\textrm{span}(x_0, x_2,\dots,x_{n-1})$.  Therefore, $\alpha(\mathcal{P})=\beta(\mathcal{P})=n-1$.

        \item $\lambda_1\neq 0$,  $\lambda_3 = 0$. In this case   $\alpha_A(\mathcal{P})=\beta_A(\mathcal{P})=n-1$ and is given  by $\textrm{span}(x_1, \dots,x_{n-1})$.   Therefore, $\alpha(\mathcal{P})=\beta(\mathcal{P})=n-1$.
 \item $\lambda_3\neq 0$,  $\lambda_1 = 0$. In this case   $\alpha_A(\mathcal{P})=\beta_A(\mathcal{P})=n-1$ and is given  by $\textrm{span}(x_0, x_2,\dots,x_{n-1})$.   But, since $[x_0,x_2]=x_3$ for instance, it can be checked that  $\alpha(\mathcal{P})=\beta(\mathcal{P})=n-2$ and is given by $\textrm{span}(x_2,\dots,x_{n-1})$.
        
    \end{enumerate}
\end{rem}

\subsection{$3$-dimensional complex Poisson algebras}

The classification of the $3$-dimensional complex Poisson algebras was given by \cite{pa3}. We have computed the invariants {$\alpha$}, {$\beta$}, {$\alpha_A$}, {$\beta_A$}, {$\alpha_L$} and {$\beta_L$} for these algebras (see Table \ref{tab1}). Trivial algebras are omitted.

{\small
\begin{longtable}{|c|c|c|c|c|c|c|c|}
\hline
{\textrm{Algebra}}  & {\textrm{Multiplication table}}  & 
{$\alpha$}  & 
{$\beta$}  & 
{$\alpha_A$}  &
{$\beta_A$}  &
{$\alpha_L$}  &
{$\beta_L$}  \\
\hline
\hline

            $\mathcal{P}_{3,14}$ & $ 
            \begin{tabular}{l}
            $e_{1}\cdot e_{1}=e_{2},$ \\ 
            $[ e_{1},e_{3}] =e_{3}.$%
            \end{tabular}%
             $ & $2$ & $2$ & $2$ & $2$& $2$& $2$ 
     
            \\ \hline $\mathcal{P}_{3,15}$ & $ 
            \begin{tabular}{l}
            $e_{1}\cdot e_{1}=e_{2},$ \\ 
            $[ e_{1},e_{3}] =e_{2}.$%
            \end{tabular}%
             $ & $2$ & $2$ & $2$ & $2$& $2$& $2$ 
     

            \\ \hline $\mathcal{P}_{3,16}^{p\neq0}$ & $ 
            \begin{tabular}{l}
            $e_{1}\cdot e_{2}=e_{3},$ \\ 
            $[ e_{1},e_{2}] =p e_{3}.$%
            \end{tabular}%
             $& $2$ & $2$ & $2$ & $2$& $2$& $2$ 
     
     
            \\ \hline $\mathcal{P}_{3,18}$ & $ 
            \begin{tabular}{l}
            $e_{1}\cdot e_{1}=e_{1},e_{1}\cdot e_{2}=e_{2},e_{1}\cdot          e_{3}=e_{3},$ \\ 
            $[ e_{2},e_{3}] =e_{2}.$%
            \end{tabular}%
             $ & $1$ & $1$ & $2$ & $2$& $2$& $2$ 
     
     
            \\ \hline $\mathcal{P}_{3,20}$ & $ 
            \begin{tabular}{l}
            $e_{1}\cdot e_{1}=e_{1},$ \\ 
            $[ e_{2},e_{3}] =e_{2}.$%
            \end{tabular}%
             $ & $1$ & $1$ & $2$ & $2$& $2$& $2$ 
            \\ \hline

   \caption{Invariants for the $3$-dimensional Poisson algebras.}
   \label{tab1}
\end{longtable}}

\subsection{$4$-dimensional nilpotent complex Poisson algebras}

The complete classification of the nilpotent complex Poisson algebras of dimension $4$ was given in \cite{pan}. We have studied the {invariants} $\alpha$, $\beta$, $\alpha_A$, $\beta_A$, $\alpha_L$ and $\beta_L$ for these algebras (see Table \ref{tab2}). Trivial algebras and split extensions are omitted.

{\small
\begin{longtable}{|c|c|c|c|c|c|c|c|}
\hline
{\textrm{Algebra}}  & {\textrm{Multiplication table}}  & 
{$\alpha$}  & 
{$\beta$}  & 
{$\alpha_A$}  &
{$\beta_A$}  &
{$\alpha_L$}  &
{$\beta_L$}  \\
\hline
\hline







$\mathcal{P}_{4,7}$ & $
\begin{tabular}{c}
$e_1\cdot e_1 = e_4,$ \\
$[ e_2, e_3 ] =e_{4}.$
\end{tabular}
$& $2$ & $2$ & $3$ & $3$& $3$& $3$ \\
\hline

$\mathcal{P}_{4,8}$ & $
\begin{tabular}{c}
$e_1\cdot e_1 = e_4, e_2\cdot e_2 = e_{4},$ \\
$[ e_1, e_3 ] =e_{4}.$
\end{tabular}
$& $2$ & $2$ & $3$ & $3$& $3$& $3$ \\
\hline

$\mathcal{P}_{4,9}$ & $
\begin{tabular}{c}
$e_1\cdot e_1 = e_4, e_2\cdot e_2 = -e_{4},$ \\
$[ e_1, e_3 ] =e_{4}, [ e_2, e_3 ] =e_{4}.$
\end{tabular}
$& $3$ & $3$ & $3$ & $3$& $3$& $3$ \\
\hline

$\mathcal{P}_{4,10}^{p}$ & $
\begin{tabular}{c}
$e_1\cdot e_2 = e_4, e_3\cdot e_3 = e_4,$ \\
$[ e_1, e_3 ] =e_{4}, [ e_2, e_3 ]=p e_{4}.$
\end{tabular}
$& $2$ & $2$ & $2$ & $2$& $3$& $3$ \\
\hline


$\mathcal{P}_{4,12}$ & $
\begin{tabular}{c}
$e_1\cdot e_1 = e_2, e_1\cdot e_2 = e_4, e_3\cdot e_3 = e_4$,\\
$[ e_1, e_3 ] =e_{4}.$
\end{tabular}
$& $2$ & $2$ & $2$ & $2$& $3$& $3$ \\
\hline


$\mathcal{P}_{4,14}$ & $
\begin{tabular}{c}
$e_1\cdot e_1 = e_2, e_1\cdot e_2 = e_4,$\\
$[ e_1, e_3 ] =e_{4}.$
\end{tabular}
$& $3$ & $3$ & $3$ & $3$& $3$& $3$ \\
\hline

$\mathcal{P}_{4,15}$ & $
\begin{tabular}{c}
$e_1\cdot e_1 = e_4, e_2\cdot e_2 = e_4,$ \\
$[ e_1, e_2 ] =e_{3}, [ e_1, e_3 ] = e_{4}.$
\end{tabular}
 $& $2$ & $2$ & $3$ & $3$& $3$& $3$ \\
\hline

$\mathcal{P}_{4,16}$ & $
\begin{tabular}{c}
$e_2\cdot e_2 = e_4,$ \\
$[ e_1, e_2 ] =e_{3}, [ e_1, e_3 ] = e_{4}.$
\end{tabular}
 $& $2$ & $2$ & $3$ & $3$& $3$& $3$ \\
\hline

$\mathcal{P}_{4,17}$ & $
\begin{tabular}{c}
$e_1\cdot e_1 = e_4,$ \\
$[ e_1, e_2 ] =e_{3}, [ e_1, e_3 ] = e_{4}.$
\end{tabular}
 $& $3$ & $3$ & $3$ & $3$& $3$& $3$ \\
\hline

$\mathcal{P}_{4,18}$ & $
\begin{tabular}{c}
$e_1\cdot e_2 = e_4,$ \\
$[ e_1, e_2 ] =e_{3}, [ e_1, e_3 ] = e_{4}.$
\end{tabular}
 $& $3$ & $3$ & $3$ & $3$& $3$& $3$ \\
\hline



$\mathcal{P}_{4,21}^{p}$ & $
\begin{tabular}{c}
$e_1\cdot e_1 = e_4, e_1 \cdot e_2 = p e_3,$ \\
$[ e_1, e_2 ] = e_3.$
\end{tabular}
 $& $3$ & $3$ & $3$ & $3$& $3$& $3$ \\
\hline

$\mathcal{P}_{4,22}$ & $
\begin{tabular}{c}
$e_1\cdot e_1 = e_4, e_2 \cdot e_2 = e_3,$ \\
$[ e_1, e_2 ] = e_3.$
\end{tabular}
 $& $2$ & $2$ & $2$ & $2$& $3$& $3$ \\
\hline



$\mathcal{P}_{4,25}$ & $
\begin{tabular}{c}
$e_1\cdot e_2 = e_4,$ \\
$[ e_1, e_2 ] = e_3.$
\end{tabular}
 $& $3$ & $3$ & $3$ & $3$& $3$& $3$ \\
\hline

$\mathcal{P}_{4,26}^{p=0}$ & $
\begin{tabular}{c}
$e_1\cdot e_1 = e_3, e_1\cdot e_2 = e_4,$ \\
$[ e_1, e_2 ] = e_3.$
\end{tabular}
 $ & $3$ & $3$ & $3$ & $3$& $3$& $3$ \\
\hline

$\mathcal{P}_{4,26}^{p\neq0}$ & $
\begin{tabular}{c}
$e_1\cdot e_1 = e_3, e_2\cdot e_2 =p e_3, e_1\cdot e_2 = e_4,$ \\
$[ e_1, e_2 ] = e_3.$
\end{tabular}
 $ & $2$ & $2$ & $2$ & $2$& $3$& $3$ \\
\hline

   \caption{Invariants for the $4$-dimensional nilpotent Poisson algebras.}
   \label{tab2}
\end{longtable}}

\bibliographystyle{amsplain}

\begin{thebibliography}{99}

\bibitem{pan}
Abdelwahab H. et al.
The algebraic classification and degenerations of nilpotent Poisson algebras. {\it J. of Algebra.} {\bf 615}, 243--277 (2023).

\bibitem{pa3}
 Abdelwahab H., 
 Fernández Ouaridi A., 
 Martín González C.,  
 Degenerations of Poisson algebras, 
 {\it Journal of Algebra and Its Applications} (2023) 
 DOI:10.1142/S0219498825500872.

\bibitem{oscillator}
Albuquerque H. et al.
Poisson algebras and symmetric Leibniz bialgebra structures on oscillator Lie algebras, {\it Journal of Geometry and Physics}, 160 (2021)

\bibitem{amayo1}
Amayo R. K.
Quasi-ideals of Lie algebras I, {\it Proc. London Math. Soc.}, (3) {\bf 33}, 28--36 (1976)

\bibitem{amayo2}
Amayo R. K.
Quasi-ideals of Lie algebras II, {\it Proc. London Math. Soc.}, (3) {\bf 33}, 37--74 (1976)


        

\bibitem{BC12} Burde,  D., Ceballos,  M., Abelian ideals of maximal dimension for solvable Lie
algebras.  {\it J. Lie Theory} {\bf 22}(3),  741--756 (2012).


\bibitem{Ceballos} Ceballos,  M., Abelian subalgebras and ideals of maximal dimension in Lie algebras. PhD dissertation. (2012) 

\bibitem{Towers13} Ceballos,  M., Towers,  D.A., On abelian subalgebras and ideals of maximal dimension in supersolvable Lie algebras.  {\it J. Pure Appl.  Algebra} {\bf 218}(3), 497--503 (2014).

\bibitem{Towers} Ceballos,  M.,  Towers,  D.A., Abelian Subalgebras and Ideals of Maximal Dimension in Solvable Leibniz Algebras.  {\it  Mediterranean Journal of Mathematics } 20:97 (2023) 




\bibitem{TowersZ} Ceballos,  M.,  Towers,  D.A., Abelian subalgebras and ideals of maximal dimension in Zinbiel algebras. {  Communications in Algebra} 51 (4) (2022)

\bibitem{Farkas1}
Farkas D. R., Poisson polynomial identities, {\it Communications in Algebra}, 26:2, 401-416, (1998). 

\bibitem{Farkas2}
Farkas D. R., Poisson polynomial identities. II, {\it Arch. Math.} (Basel) 72  252-260, (1999).



\bibitem{jacobson}
Jacobson, N. Lie algebras. Dover, New York. (1979).



\bibitem{maryam}
Mirzakhani, M., A Simple Proof of a Theorem of Schur. {\it 
 The American Mathematical Monthly}, 105(3), 260–262 (1998). 






\bibitem{Towers27} Towers, D.A., Abelian subalgebras and ideals of maximal dimension in supersolvable and nilpotent Lie algebras.  {\it Linear Multilinear Algebra} (2020).  




\bibitem{Vergne}
Vergne M., Cohomologie des algèbres de Lie nilpotentes. Application à l'étude de la variété des algèbres de Lie nilpotentes. {\it Bull. Soc. Math. France} 98, 81--116 (1970).

\end{thebibliography}

\end{document}